\documentclass[final,onefignum]{siamart190516}

\usepackage{amssymb,bbm,enumitem}
\usepackage{pgfplots,tikz}
\pgfplotsset{compat=1.14}

\newcommand\bg{\mathbf{g}}
\newcommand\bn{\mathbf{n}}
\newcommand\bq{\mathbf{q}}

\newcommand\bu{\mathbf{u}}
\newcommand\bv{\mathbf{v}}

\newcommand\by{\mathbf{y}}

\newcommand\bG{\mathbf{G}}
\newcommand\bQ{\mathbf{Q}}

\newcommand\bX{\mathbf{X}}

\newcommand{\Div}{\nabla\cdot}
\newcommand\eps{\epsilon}
\newcommand{\grad}{\nabla}

\newcommand{\ip}[2]{\ensuremath{\left<#1,#2\right>}}

\newcommand\RR{\mathbb{R}}

\title{Conservation laws for free-boundary fluid layers\thanks{Draft date: \today.  Supported by NASA grant \# NNX13AM16G.}}

\author{Ed Bueler\thanks{Dept.~of Mathematics and Statistics, University of Alaska Fairbanks \,\, (\texttt{elbueler@alaska.edu}).}}

\begin{document}
\maketitle

\begin{abstract}
Time-dependent models of fluid motion in thin layers, subject to signed source terms, represent important sub-problems within climate dynamics.  Examples include ice sheets, sea ice, and even shallow oceans and lakes.  We address these problems as discrete-time sequences of continuous-space weak formulations, namely (monotone) variational inequalities or complementarity problems, in which the conserved quantity is the layer thickness.  Free boundaries wherein the thickness and mass flux both go to zero at the margin of the fluid layer generically arise in such models.  After showing these problems are well-posed in several cases, we consider the limitations to discrete conservation in numerical schemes.  A free boundary in a region of negative source---an ablation-caused margin---turns out to be a barrier to exact conservation in either a continuous- or discrete-space sense.  We then propose computable \emph{a posteriori} quantities which allow conservation-error bounds in finite volume and finite element schemes.
\end{abstract}

\pagestyle{myheadings}
\thispagestyle{plain}
\markboth{ED BUELER}{CONSERVATION LAWS FOR FREE-BOUNDARY FLUID LAYERS}

\section{Introduction}  \label{sec:intro}

Consider a thin layer of fluid which is free to move about on a solid substrate.  Suppose that, in addition, mass can be added (accumulation, precipitation) or removed (ablation, evaporation) from the fluid layer by external processes.  Through flow and these addition/removal processes, the geometry of the layer varies in time and space.  We consider models of such fluid layers in which the layer geometry is described by a nonnegative thickness function.  In such models the addition/removal processes can be combined into a signed source term in a two-spatial-dimension mass conservation (or balance) equation.  Note that the addition/removal processes and the substrate topography are defined on a larger (fixed) region than the fluid-covered area.  Assuming the thickness function is continuous, the conservation equation applies only in the open set where the thickness is positive.  The problem of simultaneously determining the fluid motion and the fluid-covered domain is of free-boundary type.

The physics of such models couples the mass conservation equation to additional momentum and energy conservation laws.  The addition/removal processes, i.e.~the ``climate'' of the fluid layer, may also be coupled to the conservation equations, as when glacier thickness affects surface elevation and thus the precipitation rate.  Solving the resulting model, combining conservation equations, addition/removal processes, and additional closure relationships as needed, determines the nontrivial manner in which the layer geometry evolves.

This paper contains a basic, necessarily incomplete, analysis of the mathematical well-posedness of such climate-driven fluid layer models.  We start by extracting the minimal mathematical form, namely a scalar conservation equation and the nonnegative-thickness constraint.  After considering well-posedness based on several flux-form possibilities, we address tradeoffs and barriers inherent in the numerical solutions of such models.

Problems of this type appear within models of glaciers and ice sheets \cite{Bueler2016,CalvoDuranyVazquez2000,DiazSchiavi1999,
EgholmNielsen2010,JouvetBueler2012,JouvetBuelerGraeserKornhuber2013}, surface and subsurface hydrology \cite{AlonsoSantillanaDawson2008,Maxwelletal2015}, and sea ice \cite{LipscombHunke2004,Thorndikeetal1975}.  Generally, multiphysics Earth system models often contain thin-layer, free-boundary sub-models for various species (or phases) of fluids.  For example, in comprehensive models of glaciers and ice sheets there are submodels describing supra- and subglacial hydrology of liquid water \cite{Aschwandenetal2012,BuelervanPelt2015,Schoofetal2012}, floating ice shelves \cite{Albrechtetal2011}, and sediment transport \cite{Brinkerhoffetal2017}.

In such geophysical and climate-modeling contexts, determining the fluid-covered area is a leading-order modeling goal.  For example, snow and ice are much more reflective than the substrate they cover (i.e.~land or ocean), so deciding whether grid cells are ice-covered or ice-free is a significant modeling purpose.  A goal of equal importance is the conservation of mass, including a precise accounting of mass transfers to and from the modeled fluid phases.

The above geophysical applications drive the author's interest, but the situation is as familiar as the dynamics of rain droplets on a car windshield.  Precipitation, evaporation, gravity, wind stresses, and surface tension all combine to determine the evolution of the geometry of the drops and rivulets, and of the wetted and dry domains.  Note that models of such thin fluid flows often have not included any source term \cite[for example]{Kondic2003}, but those that include evaporation will require active enforcement of nonnegative layer thickness.

If the fluid is modeled as having constant density then the (nonnegative) layer thickness can be regarded as the conserved quantity, equivalent to mass per unit area.  In models for variable density fluids the vertical integral of density is the conserved quantity (in the two-dimensional conservation equation) and this variable must also be nonnegative.  For simplicity we consider the constant-density case and we call the conserved quantity ``mass'' and the corresponding nonnegative variable ``thickness''.

Now, to be more precise let us suppose that $\Omega \subset \RR^d$ is a bounded open region with regular (Lipschitz) boundary; note $d=1,2$ in cases of geophysical interest.  The layer thickness function $u(x,t)$ is defined for $x\in \Omega$ and $t \in [0,T]$.  Where there is no fluid we have $u(x,t)=0$.  The rate of flow is described by a vector flux $\bq$ and the climate (i.e.~the addition/removal processes) by a scalar, signed source term $f$; we discuss parameterizations below.

The models we consider are usually stated in strong form.  They include at least a mass conservation equation and an obvious, though sometimes-unstated, inequality constraint:
\begin{align}
u_t + \Div \bq &= f &&\text{in } \Omega \times (0,T), \text{ where } u > 0 \label{eq:massconserve} \\
u &\ge 0 &&\text{in } \Omega \times [0,T], \label{eq:constraint}
\end{align}
along with an initial condition $u(x,0)=u_0(x)\ge 0$ defined on $\Omega$.  We emphasize that conservation equation \eqref{eq:massconserve} applies only where the fluid is present ($u>0$), and not in the remainder of $\Omega$.  The situation is pictured in Figure \ref{fig:cartoon}, where positive source values ($f>0$) are pictured as downward arrows (precipitation).

\begin{figure}[ht]
\centerline{\includegraphics[width=0.8\textwidth]{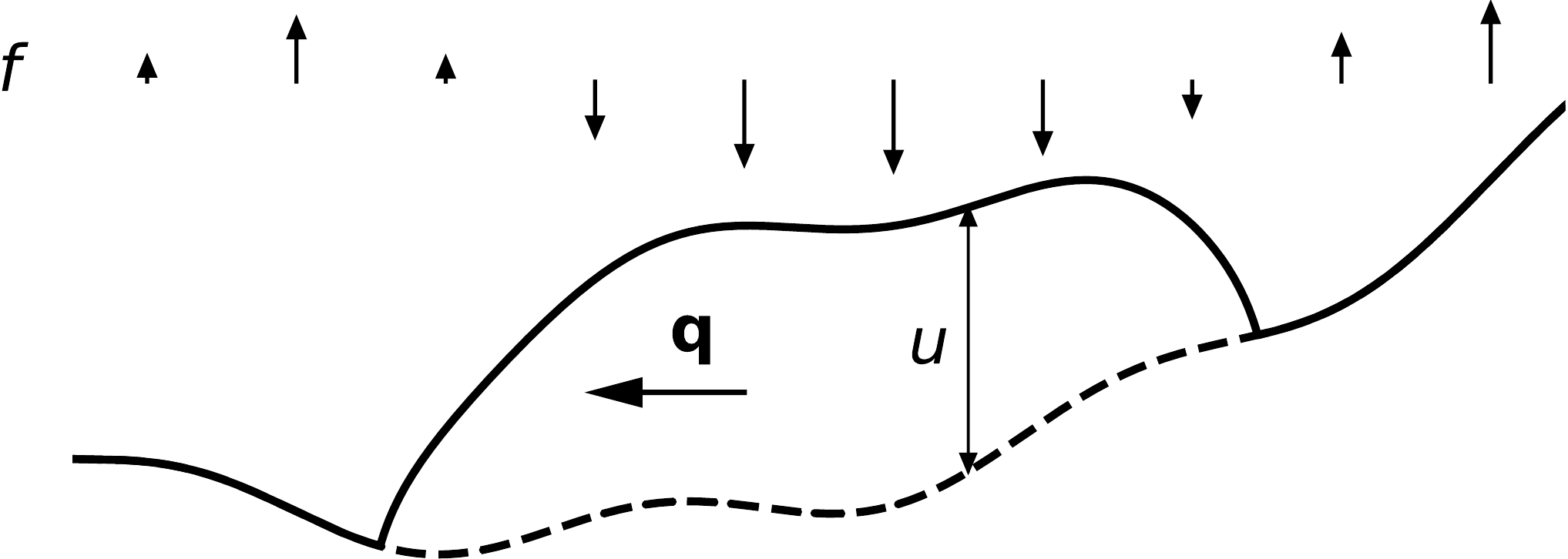}}
\caption{Schematic of a fluid layer with a thickness $u\ge 0$.}
\label{fig:cartoon}
\end{figure}

Evidently, analyzing the well-posedness of any model including \eqref{eq:massconserve} and \eqref{eq:constraint} requires additional information about $\bq$ and $f$, along with a specification of a space of admissible solutions $u$.  In most of this article we suppose that the flux $\bq$ is local, but otherwise quite general:
\begin{equation}
\bq = \bq(\grad u(x,t),u(x,t),x,t). \label{eq:fluxdepends}
\end{equation}
However, Subsection \ref{subsec:nonlocal} considers models where $\bq$ depends non-locally on integrals of $u$ over $\Omega$.  In many realistic models, computing this non-local dependence involves solving coupled differential equations.

In \eqref{eq:fluxdepends}, dependence of the flux on thickness is to be expected---thicker layers move more mass---as is dependence on $x$ because of substrate variations \cite[for example]{Bueler2016}.  The flux may additionally depend on $\grad u$ in flows which are gravity-driven and viscous; such flows are at least partly diffusive.  In simple cases the flux might be written in the form $\bq=- D \grad u + \bq_a$ where $D > 0$ has various dependence on $t,x,u,|\grad u|$---see Subsections \ref{subsec:plap} and \ref{subsec:powertransform} below---with the advective flux $\bq_a$ perhaps independent of $\grad u$.  In fact, equation \eqref{eq:massconserve} may be dominantly advective.  In the simplest advective case mass moves at some vertically-averaged velocity $\bX=\bX(x,t)$ determined by external factors, and we then have $\bq_a = \bX u$; see Subsection \ref{subsec:advect}.  In such cases we will add a small diffusion term to establish well-posedness.  Our results for such advective fluxes will apply even if $\bX$ comes from a (coupled) solution of a momentum conservation system, for example, as long as it has the regularity needed to apply the theory (Subsection \ref{subsec:fluxassumptions}).

In \eqref{eq:massconserve} the source function $f$ is allowed to be nonlinear in $u$ because feedback between layer thickness $u$ and the source $f$ occurs in certain applications \cite{Jouvetetal2011}.  However, when proving well-posedness in Section \ref{sec:wellposed} we simplify to the $u$-independent case $f=f(x,t)$; thus we do not address the impact of ``reaction'' type processes on well-posedness.

Numerical simulations of these fluid layers necessarily discretize time in some manner.  Section \ref{sec:strongform} considers time semi-discretizations of the mass conservation equation by implicit one-step methods.  (This is the method-of-lines in the orthogonal sense from the usual.)  In Section \ref{sec:weakform} we pose the continuous-space problem for a single time step in weak variational form so each time-step requires the solution of a (continuous) free-boundary problem in space.

An immediate question is:
  \begin{quote}
  \renewcommand{\labelenumi}{(\roman{enumi})}
  \begin{enumerate}
  \item \emph{Is a single time-step free-boundary problem well-posed?}
  \end{enumerate}
  \end{quote}
The answer to (i) depends on the form of the flux, but by examining a weak form and using the theory of monotone variational inequalities \cite{KinderlehrerStampacchia1980} we can show that the answer is often ``yes'' (Section \ref{sec:wellposed}).  However, even implicit cases, our sufficient conditions sometimes require a time-step restriction.

A second question is equally important in modeling practice:
  \begin{quote}
  \renewcommand{\labelenumi}{(\roman{enumi})}
  \begin{enumerate}
  \setcounter{enumi}{1}
  \item \emph{Can the mass of the fluid layer be conserved exactly in the sense that a computable space-time integral of the source term $f$ is equal to the change in mass during a time step?}
  \end{enumerate}
  \end{quote}
(This question makes sense when the answer to (i) is ``yes.'')  By considering question (ii) abstractly in Section \ref{sec:timeseries} we conclude that the answer is often ``no.''  In general a numerical model of a fluid layer governed by \eqref{eq:massconserve} and \eqref{eq:constraint} \emph{cannot} exactly conserve mass when the free boundary moves during a time-step.  Specifically, discrete-time conservation fails when margin retreat occurs, as generated by a negative source term ($f<0$).

We may, however, bound and report the mass conservation error in a practical manner.  Quantification of conservation errors in free-boundary models is a major purpose which guides the structure of this paper.  Of course, exact discrete conservation within the fluid, i.e.~away from any free boundaries, is a common goal, and property, of numerical schemes \cite[and references therein]{LeVeque2002}.  When we consider fully-discretized models in Section \ref{sec:spacediscretized} we will indeed \emph{assume} such exact discrete conservation in the interior of the fluid-covered domain.  The discrete conservation barriers we identify are thus entirely at the free boundary, and they are only active within negative source term areas.

Theoretical guidance as to achievable discrete conservation is generally absent in the literature of these free-boundary fluid problems.  Reference \cite{IdelsohnOnate2010} addresses a related conservation challenge at the free surfaces of fluids but the problem is not free-boundary in the same map-plane sense.  In the context of glacier \cite{JaroschSchoofAnslow2013} and ice shelf \cite{Albrechtetal2011} modeling, schemes for improved discrete mass conservation at free boundaries are proposed, but this small literature provides only \emph{ad hoc} and fully-discretized solutions.

The ideas and results in this paper are nontrivial if the source function $f(u,x,t)$ in \eqref{eq:massconserve} is sometimes negative.  If $f\ge 0$ holds everywhere then active enforcement of constraint \eqref{eq:constraint} may not be necessary because a maximum principle may imply the nonnegativity of the solution.  Indeed, we will see that there is no conservation error at the free boundary, at least in the continuous-space theory, when using a backward Euler temporal discretization, under the additional hypothesis that $f\ge 0$ in \eqref{eq:massconserve}.

Regarding the presence of a signed source term, the modeling goals of the debris flow \cite{GeorgeIverson2014} and tsunami run-up \cite{LeVequeetal2011} literature provide a useful contrast to our concerns.  These fluid-layer problems are of free-boundary type for a hyperbolic system of mass and momentum conservation equations.  The thickness $u$ of the flow must be nonnegative, and the discrete models allow wet ($u>0$) and dry ($u=0$) cells.  However, the time-scales are sufficiently short (seconds to hours) so that addition/removal sources like precipitation, evaporation, or absorption into the ground are usually absent from the conservation of mass equation; e.g.~$f=0$ in \eqref{eq:massconserve} in the models found in \cite{GeorgeIverson2014,LeVequeetal2011}.  Without such a source term the discrete-time sequence of free-boundary problems, if the model is formulated that way, call for constancy of the total mass, despite the moving boundary between wet and dry areas.  In these models nonnegative fluid-layer thickness can be preserved by maximum-principle or strong-stability properties of the discrete scheme, and exact discrete conservation can apply automatically.

The mass-conservation considerations and free-boundary techniques of the current paper could be applied to sea ice models, but subject to re-interpretation because of the manner in which the mass distribution is described in such models.  They typically track a non-negative probability distribution function $g(x,t,h)$, at each location $x$, where $h$ is the thickness dimension and $\int_0^\infty g\,dh = 1$ \cite[for example]{Thorndikeetal1975}.  Then $h$ is discretized into ``categories'' $g_k(x,t) = P\{H_{k-1} < h \le H_k\}$ with $g_0 = P\{h=0\}$ denoting the ice-free category \cite{LipscombHunke2004}.  Our results are relevant to the continuous-space equations which remain after discretization of $t$ and $h$.  In such models melting is a negative source term in the evolution equation for the $g_1$ category, thus (explicit) updating of $g_1^n(x) \approx \int_{H_0}^{H_1} g(x,t_n,h)\,dh$ requires truncation (projection) to maintain nonnegativity of $g_1^n$.  The inequality constraint $g_1^n \ge 0$ is a not-necessarily-stated, but in fact important, part of such schemes.

\section{Time semi-discretization}  \label{sec:strongform}

Let $\{t_n\}_{n=0}^N$ be a sequence of increasing times, with $t_0=0$ and $t_N=T$, and set $\Delta t_n = t_n-t_{n-1}>0$.  Corresponding to \eqref{eq:massconserve} and \eqref{eq:constraint}, the (strong form) \emph{single time-step problem} is
\begin{equation}
\frac{u_n - u_{n-1}}{\Delta t_n} + \Div \bQ_n(\grad u_n,u_n,x) = F_n(u_n,x) \qquad \text{where $u_n > 0$ in $\Omega$}, \label{eq:semimassconserve}
\end{equation}
and
\begin{equation}
u_n \ge 0 \qquad \text{at all points in } \Omega. \label{eq:semiconstraint}
\end{equation}
We expect this problem to determine a new thickness function $u_n(x) \approx u(x,t_n)$ given $u_{n-1}(x) \approx u(x,t_{n-1})$, as shown in Figure \ref{fig:timestepcartoon}.  The weak form of the problem is given in Section \ref{sec:weakform}, but we state the strong form first because of the developed intuition of most practitioners.

\begin{figure}[ht]
\begin{center}
\includegraphics[width=3.9in,keepaspectratio=true]{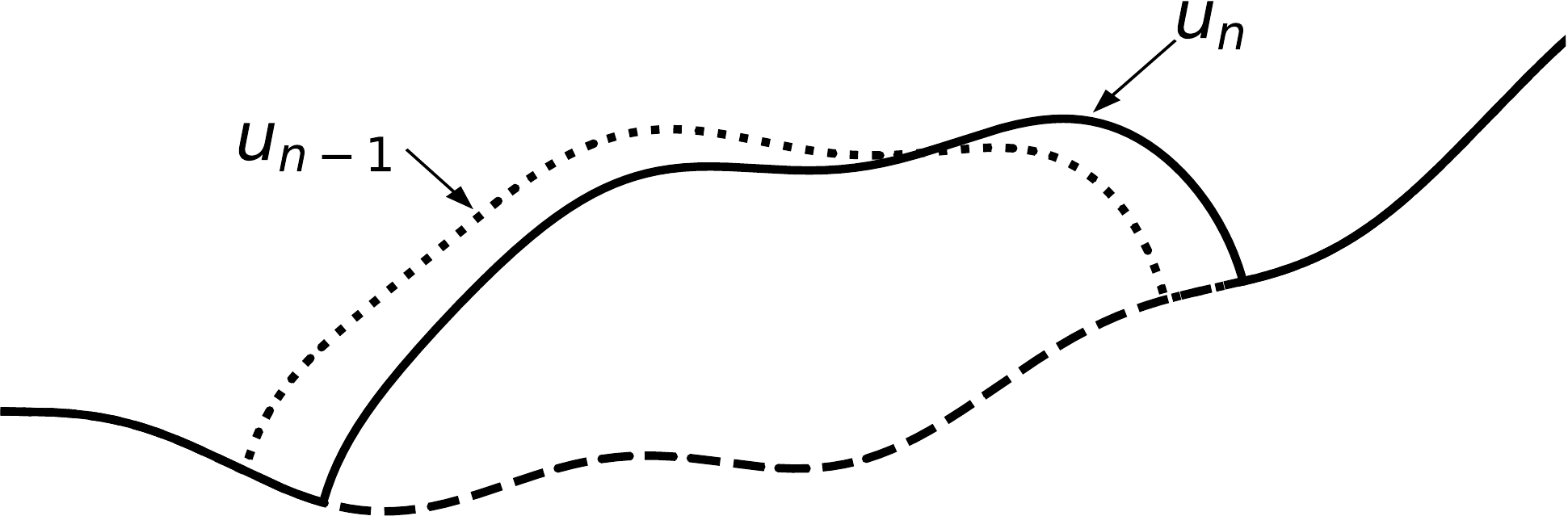}
\end{center}
\caption{The single time-step problem \eqref{eq:semimassconserve}, \eqref{eq:semiconstraint} is a free boundary problem for the new thickness $u_n\ge 0$.}
\label{fig:timestepcartoon}
\end{figure}

The semi-discretization procedure which generates equations \eqref{eq:semimassconserve} and \eqref{eq:semiconstraint}---we give examples next---corresponds to a choice of functions
\begin{equation}
\bQ_n(\bX,v,x), \quad F_n(v,x) \label{eq:functionalforms}
\end{equation}
derived from $\bq$ and $f$.  Here $\bX\in\RR^d$, $v\ge 0$, and $x\in \Omega$.  We will assume $\bQ_n$ is defined for all $x\in\Omega$, not just where $v(x)>0$.

\subsection{$\theta$ methods}  \label{subsec:thetamethods}  Consider a $\theta$-method discretization \cite{MortonMayers2005} of \eqref{eq:massconserve} with $0\le \theta \le 1$:
\begin{align}
  &\frac{u_n - u_{n-1}}{\Delta t_n} + \theta\, \Div \bq(\grad u_n,u_n,x,t_n) + (1-\theta) \Div \bq(\grad u_{n-1},u_{n-1},x,t_{n-1}) \label{eq:thetamethod} \\
  &\qquad =  \theta f(u_n,x,t_n) + (1-\theta) f(u_{n-1},x,t_{n-1}). \notag
\end{align}
Recall that the $\theta=0$ case is the forward Euler method, $\theta=1/2$ is trapezoid (Crank-Nicolson), and $\theta=1$ is backward Euler.  Equation \eqref{eq:thetamethod} is of form \eqref{eq:semimassconserve} with
\begin{align*}
\bQ_n(\bX,v,x) &= \theta\, \bq(\bX,v,x,t_n), \\
F_n(v,x)       &= \theta f(v,x,t_n) + (1-\theta) f(u_{n-1},x,t_{n-1}) \\
               &\qquad - (1-\theta) \Div \bq(\grad u_{n-1},u_{n-1},x,t_{n-1}).
\end{align*}

For any $\theta$ the source function $F_n$ ``absorbs'' all the terms which do not involve the flux $\bq$ evaluated at time $t_n$.  We will see that implicitness ($\theta>0$) is helpful both for the usual stability reasons \cite{MortonMayers2005} and to give the smoothness needed so that the weak form of \eqref{eq:semimassconserve}, \eqref{eq:semiconstraint} can be well-posed (Section \ref{sec:weakform}).  For the backward Euler scheme with $\theta=1$ observe that $\bQ_n = \bq(\bX,v,x,t_n)$ and $F_n = f(v,x,t_n)$, while if $\theta=0$ then $\bQ_n=0$ (Subsection \ref{subsec:explicit}).  Finally, such time-discretization need not be limited to $\theta$-methods; Appendix \ref{app:rk2} considers certain Runge-Kutta schemes.

\subsection{Associated set decomposition}  \label{subsec:setdecompose}  To derive the weak form, let us suppose \eqref{eq:semimassconserve} and \eqref{eq:semiconstraint} can be solved.  A solution $u_n$ then decomposes $\Omega$ into three disjoint regions:
\begin{align*}
\Omega_n &= \left\{x \in \Omega \,\big|\, u_n(x)>0\right\}, \\
\Omega_n^r &= \left\{x \in \Omega \,\big|\, u_n(x)=0 \text{ and } u_{n-1}(x) > 0\right\}, \\
\Omega_n^{00} &= \left\{x \in \Omega \,\big|\, u_n(x)=0 \text{ and } u_{n-1}(x) = 0\right\},
\end{align*}
so that
\begin{equation}
\Omega = \Omega_n \cup \Omega_n^r \cup \Omega_n^{00}.  \label{eq:omegadecomposition}
\end{equation}
Here the superscript ``$r$'' stands for ``retreat,'' and we call $\Omega_n^r$ the \emph{retreat set}.  Figure \ref{fig:domains} illustrates this decomposition.  Note that if $u_n$ and $u_{n-1}$ are continuous then $\Omega_n$ is open while $\Omega_n^{00}$ is closed (in $\Omega$).

\begin{figure}[ht]
\medskip
\begin{center}
\includegraphics[width=2.3in,keepaspectratio=true]{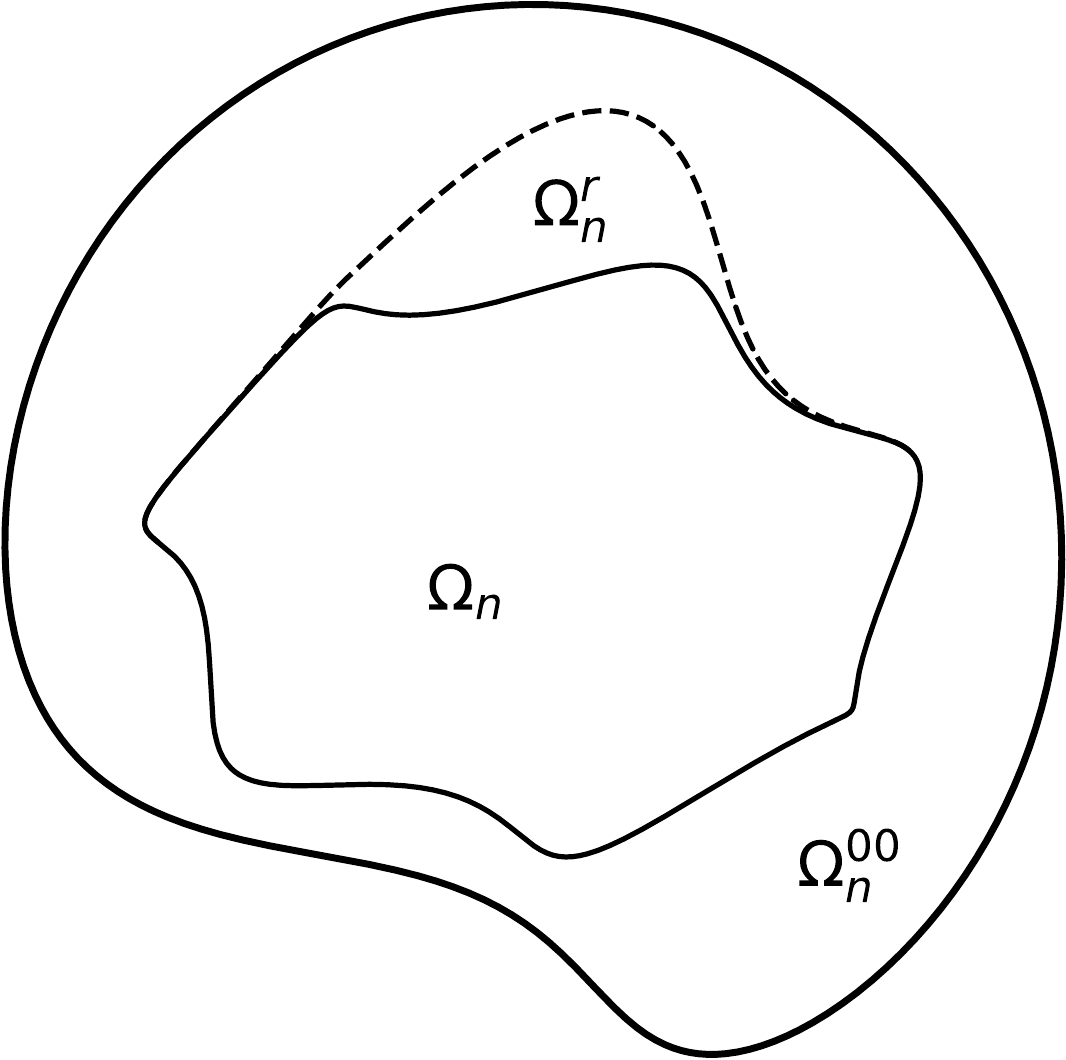}
\end{center}
\caption{A solution to \eqref{eq:semimassconserve} and \eqref{eq:semiconstraint} decomposes $\Omega$ as described in \eqref{eq:omegadecomposition}.}
\label{fig:domains}
\end{figure}

One may of course rewrite \eqref{eq:semimassconserve} as
\begin{equation}
u_n = u_{n-1} + \Delta t_n\, F_n - \Delta t_n\, \Div \bQ_n.  \label{eq:strongflat}
\end{equation}
The constraint $u_n \ge 0$ implies that the terms on the right side of \eqref{eq:strongflat} must sum to a nonnegative number.  While equation \eqref{eq:strongflat} applies where $u_n>0$, because $u_{n-1}\ge 0$ and $\Delta t_n > 0$ we expect that in the interior of $\Omega_n^r \cup \Omega_n^{00}$, where $\Div \bQ_n=0$ (see \eqref{eq:Qiszero} below), an inequality instead holds:
\begin{equation}
u_{n-1} + \Delta t_n\, F_n \le 0. \label{eq:strongconditionwherezero}
\end{equation}
Thus also $F_n \le 0$ on the same set.  Inequality \eqref{eq:strongconditionwherezero}, used below in deriving the weak form, says that the source term must be nonpositive in a zero-thickness location.

\section{Weak formulation of the single time-step problem}  \label{sec:weakform}

The form of the single time-step problem in \eqref{eq:semimassconserve}, \eqref{eq:semiconstraint} is in fact not adequate for mathematical progress.  PDE \eqref{eq:semimassconserve} applies only where its solution $u_n$ is positive, and inequality \eqref{eq:strongconditionwherezero} applies on the set where $u_n=0$, so we have ``posed'' a problem in terms of its solution.  This form is also inadequate because the boundary conditions satisfied by $u_n$ along the free boundary $\partial\Omega_n$ are not clear.  By contrast, the weak form in this section, a variational inequality \cite{Friedman1982,KinderlehrerStampacchia1980} on a convex set of admissible functions, refers only to the set $\Omega$ and its boundary $\partial\Omega$, and not to the sets in \eqref{eq:omegadecomposition}.

\subsection{Flux assumptions} \label{subsec:fluxassumptions}  We now state certain conditions on the discrete-time flux $\bQ_n$ which are sufficient to allow construction of a variational inequality (Subsection \ref{subsec:derivevi}), a smooth solution of which will also solve the strong-form problem in Subsection \ref{subsec:interior}.  Let $p\ge 1$.  Recall that the Sobolev space $W^{1,p}(\Omega)$ \cite{Evans2010} is the set of $v \in L^p(\Omega)$ satisfying $\partial_1 v,\dots,\partial_d v \in L^p(\Omega)$ and with norm
\begin{equation}
  \|v\|_{1,p} = \left(\|v\|_{L^p}^p + \sum_{i=1}^d \|\partial_i v\|_{L^p}^p\right)^{1/p}.  \label{eq:norm}
\end{equation}
If $p>d$ then $v\in W^{1,p}(\Omega)$ has a continuous representative \cite[``Morrey's inequality'']{Evans2010}, but otherwise $v$ may be discontinuous.  Denote by $W_0^{1,p}(\Omega)$ the closure of $C_c^\infty(\Omega)$ in $W^{1,p}(\Omega)$ and assume $p^{-1} + q^{-1} = 1$.

\begin{definition}  \label{ass:std}  We say $\bQ_n$ satisfies the \emph{standard flux assumptions} if
\renewcommand{\labelenumi}{\emph{\roman{enumi}})}
\begin{enumerate}
\item for each fixed $x\in \Omega$,
\begin{equation}
(\bX,z) \mapsto \bQ_n(\bX,z,x) \text{ is continuous on } \RR^d \times \RR,  \label{eq:Qiscontinuous}
\end{equation}
\item if $v \in W^{1,p}(\Omega)$ then
\begin{equation}
\bQ_n(\grad v,v,x) \in L^q(\Omega), \label{eq:QisLq}
\end{equation}
\item and
\begin{equation}
\bQ_n(\grad v,v,x)=0 \quad \text{a.e.~on } E_v = \left\{x\in\Omega\,\big|\,v(x)=0\right\}. \label{eq:Qiszero}
\end{equation}
\end{enumerate}
\end{definition}
The third assumption says that the mass flux in a zero-thickness fluid layer is zero.  Note that $\grad v = 0$ a.e.~on $E_v$ \cite[lemma A.4 in chapter II]{KinderlehrerStampacchia1980}.

Regarding the source term $F_n$ we assume only that if $v\in W^{1,p}(\Omega)$ then
\begin{equation}
F_n(v,x) \in L^q(\Omega).  \label{eq:FisLq}
\end{equation}

\subsection{A variational inequality weak formulation}  \label{subsec:derivevi}  To derive the weak form we need an extra smoothness assumption on $\bQ_n$:  For all open $S \subset \Omega$, if $v\in W^{1,p}(S)$ then
\begin{equation}
\frac{\partial}{\partial x_i} \bQ_n(\grad v,v,x) \in L^q(S). \label{eq:Qissmooth}
\end{equation}
This assumption will not be needed in later analysis of well-posedness of the resulting weak form (Section \ref{sec:wellposed}) or conservation errors (Sections \ref{sec:timeseries}--\ref{sec:spacediscretized}).

\begin{theorem} \label{thm:strongimpliesweak}  Suppose $u_n \in C(\overline{\Omega}) \cap W^{1,p}(\Omega)$ is a nonnegative function which solves \eqref{eq:semimassconserve} on $\Omega_n$ and \eqref{eq:strongconditionwherezero} on the interior of $\Omega_n^r \cup \Omega_n^{00}$.  Assume the boundaries of the sets $\Omega_n$ and $\Omega_n^r \cup \Omega_n^{00}$ in decomposition \eqref{eq:omegadecomposition} are Lipschitz, and that $\overline{\Omega_n} \subset \Omega$.  Suppose $\bQ_n$ satisfies the standard flux assumptions and \eqref{eq:Qissmooth}, $F_n$ satisfies \eqref{eq:FisLq}, and $\bQ = \bQ_n(\grad u_n,u_n,x)$ and $F=F_n(u_n,x)$ are continuous.  Then
\begin{equation}
-\int_{\Omega} \bQ \cdot \grad(v-u_n) \ge \int_{\Omega} \left(F - \frac{u_n - u_{n-1}}{\Delta t_n}\right) (v-u_n) \label{eq:morallytheVI}
\end{equation}
for any $v \in C(\overline{\Omega}) \cap W^{1,p}(\Omega)$ such that $v \ge 0$.
\end{theorem}

\begin{proof}  Let $w=v-u_n$.  Using decomposition \eqref{eq:omegadecomposition} and integration by parts,
\begin{align}
-\int_{\Omega} \bQ \cdot \grad w &= \int_{\Omega_n} (\Div \bQ) w - \int_{\partial \Omega_n} (\bQ \cdot \bn) w \label{eq:intbypartsfromstrong} \\
  &\qquad\quad + \int_{\Omega_n^r \cup \Omega_n^{00}} (\Div \bQ) w - \int_{\partial(\Omega_n^r \cup \Omega_n^{00})} (\bQ \cdot \bn) w. \notag
\end{align}
(This needs assumption \eqref{eq:Qissmooth} on the sets $S=\Omega_n$ and $S=(\Omega_n^r \cup \Omega_n^{00})^\circ$.)  Because $u_n$ is continuous it follows that $u_n=0$ on $\partial \Omega_n$ and on $\partial(\Omega_n^r \cup \Omega_n^{00})$.  Thus by continuity, \eqref{eq:Qiscontinuous}, and \eqref{eq:Qiszero} we see that $\bQ=0$ on these boundaries, so the boundary integrals in \eqref{eq:intbypartsfromstrong} are zero.  Now, by \eqref{eq:semimassconserve} on $\Omega_n$, and by \eqref{eq:Qiszero} and \eqref{eq:Qissmooth} we have $\Div \bQ=0$ a.e.~on $\Omega_n^r \cup \Omega_n^{00}$.  Thus we get
\begin{equation}
-\int_{\Omega} \bQ \cdot \grad w = \int_{\Omega_n} \left(F - \frac{u_n - u_{n-1}}{\Delta t_n}\right) w. \label{eq:equalitybeforeVI}
\end{equation}
However, by \eqref{eq:strongconditionwherezero}, $F \le 0$ on $\Omega_n^r \cup \Omega_n^{00}$.  Since also $u_n=0$, $u_{n-1}\ge 0$, and $w = v-u_n = v \ge 0$ on $\Omega_n^r \cup \Omega_n^{00}$, we have
\begin{equation}
    0 \ge \int_{\Omega_n^r \cup \Omega_n^{00}} \left(F - \frac{u_n - u_{n-1}}{\Delta t_n}\right) w. \label{eq:inequalitybeforeVI}
\end{equation}
Adding \eqref{eq:equalitybeforeVI} and \eqref{eq:inequalitybeforeVI} gives \eqref{eq:morallytheVI}.
\end{proof}

While this derivation of inequality \eqref{eq:morallytheVI} requires many hypotheses, it adequately motivates a weak formulation, as follows.  Fix $p>1$ and denote $\mathcal{X} = W_0^{1,p}(\Omega)$, with dual space $\mathcal{X}'$ and pairing $\ip{\cdot}{\cdot}: \mathcal{X}' \times \mathcal{X} \to \RR$.

\begin{definition}  The set of admissible layer thicknesses is the following closed and convex subset of $\mathcal{X}$:
\begin{equation}
\mathcal{K} = \left\{v \in \mathcal{X} \,\big|\, v(x) \ge 0\, \text{ a.e.~} x \in \Omega\right\}.  \label{eq:defineK}
\end{equation}
\end{definition}

\begin{definition}  Suppose $u_{n-1}\in\mathcal{K}$ and $\Delta t_n>0$.  Assume that $\bQ_n$ satisfies the standard flux assumptions and that $F_n$ satisfies \eqref{eq:FisLq}.  Define $A_n:\mathcal{K} \to \mathcal{X}'$ by
\begin{equation}
  \ip{A_n(v)}{\phi} = \int_\Omega \left(v - \Delta t_n\, F_n(v,x) - u_{n-1}\right)\phi - \Delta t_n\, \bQ_n(\grad v,v,x) \cdot \grad\phi. \label{eq:defineAn}
\end{equation}
\end{definition}

\begin{definition}  We say $u_n\in\mathcal{K}$ \emph{solves the (weak) time-step problem} if
\begin{equation}
  \ip{A_n(u_n)}{v-u_n} \ge 0 \quad \text{for all } v \in \mathcal{K}.  \label{eq:theVI}
\end{equation}
\end{definition}

Variational inequality (VI) \eqref{eq:theVI} is the same as \eqref{eq:morallytheVI}.

\subsection{Interior condition}  \label{subsec:interior}  We now prove a converse of Theorem \ref{thm:strongimpliesweak} which makes no regularity assumptions on the set decomposition \eqref{eq:omegadecomposition}.

\begin{theorem} \label{thm:weakimpliesstrong}  Assume $F_n$ satisfies \eqref{eq:FisLq} and $\bQ_n$ satisfies the standard flux assumptions plus \eqref{eq:Qissmooth}.  Choose $u_{n-1}\in\mathcal{K}$, and suppose that $u_n\in\mathcal{K}$ solves \eqref{eq:theVI}.
\renewcommand{\labelenumi}{(\roman{enumi})}
\begin{enumerate}
\item If $S \subset \Omega_n$ is open, $\overline{S}\subset \Omega$, and $u_n$ is continuous on $S$ then \eqref{eq:semimassconserve} applies a.e.~on $S$.
\item If $S \subset \Omega_n^r \cup \Omega_n^{00}$ is open then \eqref{eq:strongconditionwherezero} applies a.e.~on $S$.
\end{enumerate}
\end{theorem}

\begin{proof}  Let $\bQ = \bQ_n(\grad u_n,u_n,x)$ and $F=F_n(u_n,x)$.  Supposing $S \subset \Omega_n$ is open, choose any $\phi\in C_c^\infty(S)$ and extend it by zero to all of $\Omega$; note that $\phi$ can have either sign, but that $\phi=0$ on $\partial\Omega$.  Let $v = u_n + \eps \phi$ and note that $v \in \mathcal{K}$ as long as $\eps\in\RR$ is sufficiently small in magnitude.  (Specifically, if $|\eps|\le \eps_0 = \min u_n(x) / \max |\phi(x)|$, with the minimum and maximum taken over the closure of the support of $\phi$, then $v \in \mathcal{K}$.)  It follows from \eqref{eq:theVI} that
   $$\eps \int_\Omega \left(u_n - \Delta t_n\,F - u_{n-1}\right)\phi - \Delta t_n\,\bQ \cdot \grad \phi \ge 0.$$
This is true for all sufficiently-small $\eps$, of either sign, and thus the integral is zero.  Integration by parts, using assumption \eqref{eq:Qissmooth} and $\phi\big|_{\partial\Omega}=0$, gives
   $$\int_\Omega \left[ u_n - \Delta t_n\,F - u_{n-1} + \Delta t_n\,\Div\bQ \right]\phi = 0.$$
Because $\phi\in C_c^\infty(S)$ is arbitrary, the quantity in square brackets is zero a.e.~on $S$, i.e.~\eqref{eq:semimassconserve}, which proves \emph{(i)}.

Now suppose $S \subset \Omega_n^r \cup \Omega_n^{00}$.  Choose any \emph{nonnegative} $\phi\in C_c^\infty(S)$, extend it by zero, and let $v = u_n + \phi$ so $v\in\mathcal{K}$.  Note $u_n=0$ on the support of $\phi$.  By assumptions \eqref{eq:Qiscontinuous} and \eqref{eq:Qiszero}, $\bQ=0$ on the support of $\phi$.  Thus by \eqref{eq:theVI},
    $$0 \ge \int_{\Omega} \left(u_{n-1} + \Delta t_n\, F\right) \phi,$$
and it follows that \eqref{eq:strongconditionwherezero} a.e.~on $S$. \end{proof}

Thus, under a regularity assumption \eqref{eq:Qissmooth} on $\bQ_n$, a solution of \eqref{eq:theVI} solves PDE \eqref{eq:semimassconserve} where it is positive, but where it is zero inequality \eqref{eq:strongconditionwherezero} holds.  From now on we will use set decomposition \eqref{eq:omegadecomposition} only when referring to a solution $u_n$ of the weak form \eqref{eq:theVI}, and we will not need assumption \eqref{eq:Qissmooth}.

\section{Well-posedness of the time-step problem} \label{sec:wellposed}

We show in this Section that a variety of different fluxes $\bQ_n$ yield well-posed VI problems \eqref{eq:theVI}.  Later, the \emph{a posteriori} analysis of conservation errors in Sections \ref{sec:timeseries} and \ref{sec:spacediscretized} will assume that \eqref{eq:theVI} is well-posed.

Techniques for proving well-posedness of VIs in Banach spaces are relatively well-established for linear and some nonlinear elliptic problems, especially for monotone operators \cite{KinderlehrerStampacchia1980}, which we recall next.  Thereby we prove existence and uniqueness of the solution to \eqref{eq:theVI} for certain flux cases in these Subsections:
\begin{itemize}
\item[\ref{subsec:plap}] $p$-Laplacian-type parabolic (diffusion) for $1<p<\infty$,
\item[\ref{subsec:powertransform}] doubly-nonlinear parabolic, including porous media,
\item[\ref{subsec:advect}] linear advective, with small added diffusion term, and
\item[\ref{subsec:nonlocal}] linear and non-local, computed by integrals over $\Omega$.
\end{itemize}
These subsections only use the backward Euler time-stepping discretization, but the results can be extended to implicit $\theta$-methods, for example.  At the end, Subsection \ref{subsec:explicit} shows that if time-stepping is explicit then regularity issues generally block these time-step problems from being well-posed.

\subsection{Monotone variational inequalities} \label{subsec:mono}  Assume that $\mathcal{K}$ is any closed and convex subset of a Banach space $\mathcal{X}$.  The following definitions can be found in \cite{KinderlehrerStampacchia1980}.  A mapping $A : \mathcal{K} \to \mathcal{X}'$ is \emph{monotone} if, for all $u,v\in\mathcal{K}$,
\begin{equation}
   \ip{A(u) - A(v)}{u-v} \ge 0.  \label{eq:monodef}
\end{equation}
(This holds if $A$ is linear and positive semi-definite.)  It is \emph{strictly monotone} if equality in \eqref{eq:monodef} implies $u=v$.  Mapping $A$ is \emph{coercive} if there is $\phi\in \mathcal{K}$ so that
\begin{equation}
   \lim_{\|u\|_{\mathcal{X}}\to\infty} \frac{\ip{A(u) - A(\phi)}{u-\phi}}{\|u-\phi\|_{\mathcal{X}}} = +\infty, \label{eq:coercivedef}
\end{equation}
where the limit is taken over $u\in\mathcal{K}$.  Finally, a mapping $A : \mathcal{K} \to \mathcal{X}'$ is \emph{continuous on finite-dimensional subspaces} if for each finite-dimensional subspace $\mathcal{M} \subset \mathcal{X}$ the restriction $A : \mathcal{K}\cap \mathcal{M} \to \mathcal{X}'$ is weakly-continuous.  The theory of monotone VIs in Banach spaces \cite[chapter III]{KinderlehrerStampacchia1980} shows that a solution to a VI like \eqref{eq:theVI}, namely $\ip{A(u)}{v-u} \ge 0$ for all $v\in\mathcal{K}$, exists and is unique if $A$ is strictly monotone, coercive, and continuous on finite-dimensional subspaces.

Consider \eqref{eq:theVI} with $A_n$ defined by \eqref{eq:defineAn}.  It is easy to show the following lemma.

\begin{lemma}  \label{lem:continuous}  Assume \eqref{eq:Qiscontinuous} for $\bQ_n$ and that $F_n(v,x)\in L^q$ is continuous in $v$.  The map $A_n$ is continuous on finite-dimensional subspaces.
\end{lemma}


Now we want to relate the properties of the flux $\bQ_n$ to the monotonicity and coercivity of $A_n$.  From \eqref{eq:defineAn} the following calculation applies when $F_n=F_n(x)$, i.e.~when the source function is independent of the thickness $v$:
\begin{align}
   &\ip{A_n(u) - A_n(v)}{u-v}  \label{eq:AtoQcalculation} \\
   &\qquad = \int_\Omega (u-v)^2 - \Delta t_n\, \left[\bQ_n(\grad u,u,x) - \bQ_n(\grad v,v,x)\right] \cdot \grad(u-v).  \notag
\end{align}
Observe that in cases where $\bQ_n(\grad u,u,x)$ is proportional to $\grad u$ we expect that, for usable models, the flux $\bQ_n$ points generally in the direction of the negative of $\grad u$.  (Otherwise PDE \eqref{eq:massconserve} would behave as the ill-posed backward heat equation.)

The proof of the following lemma is an easy consequence of \eqref{eq:AtoQcalculation} (and is omitted).  Note that $W^{1,p}(\Omega) \subset L^2(\Omega)$ if either $p>d$ or $d\le 2$ \cite[theorems 5.6.2 and 5.6.5]{Evans2010}.

\begin{lemma}  \label{lem:monotonecoercive}  Suppose $W^{1,p}(\Omega) \subset L^2(\Omega)$.  Suppose \eqref{eq:QisLq} and that $F_n=F_n(x) \in L^q(\Omega)$.  Then
\renewcommand{\labelenumi}{(\roman{enumi})}
\begin{enumerate}
\item  $A_n$ is monotone if there is $C\le 1$ so that, for all $u,v \in \mathcal{K}$,
\begin{equation}
\int_\Omega \left[\bQ_n(\grad u,u,x) - \bQ_n(\grad v,v,x)\right] \cdot \grad(u-v) \le \frac{C}{\Delta t_n} \|u-v\|_{L^2}^2. \label{eq:Qnmonotone}
\end{equation}
\item  $A_n$ is strictly-monotone if \eqref{eq:Qnmonotone} holds with $C<1$,
\item  $A_n$ is coercive if there is $c>0$ and $r>1$ so that, for all $u,v \in \mathcal{K}$,
\begin{equation}
\int_\Omega \left[\bQ_n(\grad u,u,x) - \bQ_n(\grad v,v,x)\right] \cdot \grad(u-v) \le - c \|u-v\|_{1,p}^r. \label{eq:Qncoercive}
\end{equation}
\end{enumerate}
\end{lemma}

Inequality \eqref{eq:Qncoercive} implies \eqref{eq:Qnmonotone} with $C=0$, so \eqref{eq:Qncoercive} also implies strict-monotonicity for $A_n$ independently of $\Delta t_n$.  In fact \eqref{eq:Qnmonotone} is necessary and sufficient for monotonicity of $A_n$, while \eqref{eq:Qncoercive} is only sufficient for coercivity.  (For example, if the right side of \eqref{eq:Qncoercive} were $- c \|u-v\| \log \|u-v\|$ then $A_n$ would be coercive.)  Corollary III.1.8 of \cite{KinderlehrerStampacchia1980} now gives the following theorem.

\begin{theorem}  \label{thm:monowellposed}  Suppose $W^{1,p}(\Omega) \subset L^2(\Omega)$, $\bQ_n$ satisfies the standard flux assumptions, and $F_n=F_n(x)\in L^q(\Omega)$.  If \eqref{eq:Qncoercive} then the single time-step problem \eqref{eq:theVI} has a unique nonnegative solution $u\in\mathcal{K} \subset \mathcal{X}=W_0^{1,p}(\Omega)$.
\end{theorem}

\subsection{$p$-Laplacian fluxes} \label{subsec:plap}  We can apply Theorem \ref{thm:monowellposed} to show well-posedness in certain linear and non-linear parabolic cases.  First consider the $p$-Laplacian \cite{Evans2010} flux
\begin{equation}
  \bQ_n(\grad u) = - k |\grad u|^{p-2} \grad u \label{eq:plapflux}
\end{equation}
with $k>0$ and $1<p<\infty$, so that $\bQ_n$ satisfies the standard flux assumptions.  Formula \eqref{eq:plapflux} includes the linear (Fourier/Fick's law) flux as the $p=2$ case.  For the proofs in this subsection we assume $F_n=F_n(x)$ is independent of $u$, and we use inequalities from Appendix \ref{app:pinequalities}.  In the next four subsections we denote $\|\cdot\|$ for $\|\cdot\|_{1,p}$.

\begin{theorem}  \label{thm:plapwellposed}  If $\Omega\subset \RR^d$ is bounded, $1<p<\infty$, $F_n=F_n(x)$ is independent of $u$, and $\bQ_n$ is given by \eqref{eq:plapflux} with $k>0$, then for any $\Delta t_n>0$, \eqref{eq:theVI} has a unique solution $u\in\mathcal{K}$.
\end{theorem}

\begin{proof} If $p\ge 2$ then by \eqref{eq:pbiginequality} and \eqref{eq:poincare} there is $C=C(\Omega,p)$ so that
    $$\int_\Omega \left(\bQ_n(\grad u) - \bQ_n(\grad v)\right)\cdot (\grad u - \grad v) \le - \frac{k}{2^{p-2}} \int_\Omega |\grad u - \grad v|^p \le - \frac{k}{C 2^{p-2}} \|u-v\|^p$$
and thus \eqref{eq:Qncoercive} holds with $r=p$.  However, if  $1<p<2$ then we have to work harder.  Coercivity can be shown, but not via \eqref{eq:Qncoercive}.  Using \eqref{eq:psmallinequality}, \eqref{eq:smallpbound}, and \eqref{eq:poincare} gives
\begin{align*}
\ip{A_n(u) - A_n(v)}{u-v} &\ge \|u-v\|_{L^2}^2 + \Delta t_n\,k (p-1) \int_\Omega \frac{|\grad u - \grad v|^2}{\left(|\grad u|+|\grad v|\right)^{2-p}} \\
  &\ge \|u-v\|_{L^2}^2 + \Delta t_n\,k (p-1) \frac{\|\grad u - \grad v\|_{L^p}^2}{\big\||\grad u|+|\grad v|\big\|_{L^p}^{2-p}} \\
  &\ge \|u-v\|_{L^2}^2 + B\, \frac{\|u - v\|^2}{\big\||\grad u|+|\grad v|\big\|_{L^p}^{2-p}}
\end{align*}
where $B = \Delta t_n\,k (p-1) C(\Omega,p)^{-2/p} >0$.  This shows $\ip{A_n(u) - A_n(v)}{u-v} \ge \|u-v\|_{L^2(\Omega)}^2$, thus $A_n$ is strictly-monotone.  Fixing $v$ such that $\|\grad v\|_{L^p} >0$, we have
\begin{equation*}
\frac{\ip{A_n(u) - A_n(v)}{u-v}}{\|u-v\|} \ge B\, \frac{\|u - v\|}{\big\||\grad u|+|\grad v|\big\|_{L^p}^{2-p}} \to \infty
\end{equation*}
as $\|u\|\to\infty$, because $0<2-p<1$, and thus $A_n$ is coercive. \end{proof}

\subsection{Doubly-nonlinear fluxes} \label{subsec:powertransform}  Now consider the flux formula
\begin{equation}
  \bQ_n(\grad u,u) = - k u^r |\grad u|^{p-2} \grad u \label{eq:doubleflux}
\end{equation}
where $k>0$, $r\ge 0$, and $1<p<\infty$.  This includes, as the $r=0$ case, the $p$-Laplacian \eqref{eq:plapflux}, but it also includes the porous medium equation \cite{Vazquez2007}, where $p=2$, $r=\gamma-1$, and thus $\bQ_n = - k \gamma \grad(u^\gamma)$.  The flux for the diffusive shallow water equations \cite{AlonsoSantillanaDawson2008}, which has nontrivial powers $1<r<2$ and $1<p\le 2$, is also included, and the flat-bed shallow ice approximation \cite{CalvoDuranyVazquez2000} flux with $r=n+2$ and $p=n+1$ for $n>1$.

Leaving the function space undetermined for a moment, we apply a power transformation $u = w^m$ where $m = (p-1)/(r+p-1)$ \cite{Raviart1970} so that $0 < m \le 1$.  Straightforward calculation turns \eqref{eq:doubleflux} into
\begin{equation}
	\bQ_n = - K |\grad w|^{p-2} \grad w, \label{eq:doublenewflux}
\end{equation}
with $K=k m^{p-1}>0$, giving the $p$-Laplacian flux \eqref{eq:plapflux}.  This transformation converts PDE \eqref{eq:semimassconserve} into a $p$-Laplacian equation, but with additional zeroth-order terms,
\begin{equation}
    - \Div\left(K|\grad w|^{p-2} \grad w\right) + G(w,x) = 0  \label{eq:doubleplap}
\end{equation}
where $G(w,x) = w^m - \Delta t_n\, F_n(w^m,x) - u_{n-1}$.  In the porous media $p=2$ case equation \eqref{eq:doubleplap} is semilinear.

Define $\mathcal{X} = W_0^{1,p}(\Omega)$ and $\mathcal{K} =\left\{u\ge 0\right\} \subset \mathcal{X}$ as before.  Define $A_n: \mathcal{K} \to \mathcal{X}'$ by
\begin{equation}
\ip{A_n(w)}{\phi} = \int_\Omega \Delta t_n\, K |\grad w|^{p-2} \grad w\cdot \grad \phi + G(w,x)\phi. \label{eq:doubleform}
\end{equation}
The weak formulation of \eqref{eq:doubleplap} is VI \eqref{eq:theVI} but with \eqref{eq:doubleform} defining $A_n$.  The following Theorem uses the argument in subsection III.3 of \cite{KinderlehrerStampacchia1980}.

\begin{theorem}
Let $1<p<\infty$, $r\ge 0$, and define $m = (p-1)/(r+p-1)$.  Suppose $G(w,x)=w^m - \Delta t_n\, F_n(w^m,x) - u_{n-1}$ is in $\mathcal{X}'$ for all $w\in\mathcal{K}$, and that $G$ is nondecreasing in $w$.  Then $A_n$ in \eqref{eq:doubleform} is strictly monotone and coercive, and thus \eqref{eq:theVI} has a unique solution $u\in\mathcal{K}$.
\end{theorem}

\begin{proof}
Suppose $p\ge 2$.  If $w,v\in\mathcal{X}$ then by \eqref{eq:pbiginequality} and Poincare inequality \eqref{eq:poincare},
\begin{align*}
\ip{A_n(w)-A_n(v)}{w-v} &= \int_\Omega \Delta t_n\, K \left(|\grad w|^{p-2} \grad w - |\grad v|^{p-2} \grad v\right) \cdot \grad (w-v) \\
  &\qquad\qquad + \left(G(w,x) - G(v,x)\right) (w-v) \\
  &\ge \frac{\Delta t_n\,K}{2^{p-2}} \int_\Omega |\grad (w-v)|^p + 0 \ge \frac{\Delta t_n\,K}{2^{p-2}\, C(\Omega,p)} \|w-v\|^p.
\end{align*}
The case $1<p<2$ follows by modification of the argument in Theorem \ref{thm:plapwellposed}.
\end{proof}

\subsection{Advection by a differentiable velocity field} \label{subsec:advect}  The flux in certain applications (ice shelves, sea ice) is understood to be dominantly advective.  In fact the velocity arises from solving a coupled momentum balance, but here we simply assume the layer is transported by a differentiable velocity field $\bX \in W^{1,\infty}(\Omega;\RR^d)$,
\begin{equation}
  \bQ_n(u,x) = \bX(x) u. \label{eq:advectflux}
\end{equation}
If $1<p<\infty$ then $\bQ_n$ satisfies the standard flux assumptions on $W^{1,p}(\Omega)$.

Suppose $u,v\in\mathcal{K}$.  Noting $u=v=0$ on $\partial \Omega$, integration-by-parts shows
\begin{align}
   &\int_\Omega \left[\bQ_n(u,x) - \bQ_n(v,x)\right] \cdot \grad (u - v) = \int_\Omega \bX (u-v) \cdot \grad (u - v)   \label{eq:advectQnmono} \\
   &\qquad\qquad = \frac{1}{2}\,\int_\Omega \bX \cdot \grad\left[ (u - v)^2 \right] = - \frac{1}{2}\,\int_\Omega \left(\Div\bX\right) (u - v)^2. \notag
\end{align}
Equation \eqref{eq:advectQnmono} can be exploited in a couple of ways.  If the vector field is divergent $\Div\bX\ge 0$ then \eqref{eq:Qnmonotone} applies with $C=0$ and so $A_n$ is strictly monotone.  Otherwise, \eqref{eq:Qnmonotone} applies with $C = \frac{1}{2}\Delta t_n\,\|(\Div\bX)_{-}\|_{L^\infty(\Omega)}$, and then $A_n$ is monotone if $C \le 1$ (strictly if $C < 1$).

Consider the operator $A_n$ defined by \eqref{eq:defineAn} using flux \eqref{eq:advectflux}.  Unfortunately, there is no reason to suppose this operator is coercive, so we add a bit of diffusion in the form of a $p$-Laplacian leading-order term with coefficient $\eps>0$, namely
\begin{equation}
  \bQ_n(\grad u,u,x) = -\eps |\grad u|^{p-2} \grad u + \bX(x) u.   \label{eq:plapadvectflux}
\end{equation}

\begin{theorem}  \label{thm:plapadvectwellposed}  Suppose $\bQ_n$ is given by \eqref{eq:plapadvectflux} with $\eps>0$ and $p\ge 2$.  If $\bX \in W^{1,\infty}(\Omega;\RR^d)$ and $F_n=F_n(x)$ is independent of $u$ then \eqref{eq:theVI} has a unique solution $u\in\mathcal{K}$ if either $\Div \bX \ge 0$ or
\begin{equation}
  \Delta t_n \le \frac{2}{\|(\Div \bX)_-\|_{L^\infty}}. \label{eq:plapadvectdtcond}
\end{equation}
\end{theorem}

\begin{proof} Recalling the proof of Theorem \ref{thm:plapwellposed}, equation \eqref{eq:advectQnmono} gives
\begin{equation*}
\int_\Omega \left[\bQ_n(u,x) - \bQ_n(v,x)\right] \cdot \grad (u - v) \le - \eps c_0 \|u-v\|^p - \frac{1}{2} \int_\Omega (\Div\bX) (u-v)^2
\end{equation*}
where $c_0=2^{2-p}/C(\Omega,p)>0$.  If $c_1 = \eps c_0 \Delta t_n$ then
\begin{align*}
\ip{A_n(u) - A_n(v)}{u-v} &\ge c_1\, \|u-v\|^p + \frac{\Delta t_n}{2} \int_\Omega (\Div\bX) (u-v)^2 + \|u-v\|_{L^2}^2 \\
   &\ge c_1\, \|u-v\|^p + \left(1 - \frac{\Delta t_n}{2} \|(\Div \bX)_-\|_{L^\infty}\right) \|u-v\|_{L^2}^2,
\end{align*}
thus $A_n$ is coercive under either hypothesis. \end{proof}

Condition \eqref{eq:plapadvectdtcond} is independent of $\eps>0$, suggesting that the pure advection problem ($\eps = 0$) may also be well-behaved, but our monotone VI technique in $\mathcal{K} \subset W^{1,p}$ does not establish it.  Note that bound \eqref{eq:plapadvectdtcond} might be regarded as a CFL-type condition \cite{LeVeque2002}, but it measures the \emph{convergence} of the velocity field, not its magnitude.  If $\|(\Div \bX)_-\|_{L^\infty}$ is small then large time steps are well-posed.

\subsection{Non-local dependence through an integral kernel} \label{subsec:nonlocal}   The examples so far compute the flux $\bQ_n$ at $x\in\Omega$ using only the values $u(x)$ and $\grad u(x)$.  However, the flux in realistic models often comes from solving coupled differential equations, generally including momentum conservation.  In that context the flux is non-locally determined from the layer thickness $u$ and its spatial derivatives.

Let $\mathcal{X} = W_0^{1,2}(\Omega)$, a Hilbert space, and $\mathcal{K}=\{u\in\mathcal{X}|u\ge 0\}$.  Suppose $G_1(x,y)$, \dots, $G_d(x,y)$ and $K(x,y)$ are scalar, real-valued kernel functions in $L^2(\Omega\times \Omega)$.  Define
\begin{equation}
\bQ_n[\grad u,u](x) = \int_\Omega \bG(x,y) u(y)\,dy - \int_\Omega K(x,y) \grad u(y)\,dy, \label{eq:nonlocaldefQn}
\end{equation}
where $\bG(x,y) = (G_1(x,y), \dots, G_d(x,y))$ is $\RR^d$-valued.  With flux \eqref{eq:nonlocaldefQn}, equation \eqref{eq:semimassconserve} is no longer a PDE, but rather a linear integro-differential equation \cite{PorterStirling1990}.

Let $A_n:\mathcal{K} \to \mathcal{X}'=\mathcal{X}$ be defined by \eqref{eq:defineAn}, with $\bQ_n$ from \eqref{eq:nonlocaldefQn}, that is
\begin{align}
    \ip{A_n(v)}{\phi} &= \int_\Omega \bigg[(v - \Delta t_n\,F_n - u_{n-1})\phi - \Delta t_n\,\left(\int_\Omega \bG(\cdot,y) v(y)\,dy\right)\cdot \grad \phi \label{eq:nonlocaldefnAn} \\
                      &\qquad\qquad + \Delta t_n\,\left(\int_\Omega K(\cdot,y) \grad v(y)\,dy\right)\cdot \grad \phi\,\bigg]. \notag
\end{align}

\begin{theorem}  \label{thm:nonlocalwellposed}  Suppose $F_n=F_n(x) \in L^2(\Omega)$ is independent of $u$.  Assume $G_i \in L^2(\Omega\times\Omega)$ for $i=1,\dots,d$.  Suppose $K \in L^2(\Omega\times\Omega)$ is positive and bounded below in the sense that there is $\delta>0$ so that
\begin{equation}
   \int_\Omega \int_\Omega K(x,y) \phi(x) \phi(y)\,dx\,dy \ge \delta \|\phi\|_{L^2}^2 \qquad \text{for all } \phi \in L^2(\Omega).  \label{eq:nonlocalKpos}
\end{equation}
If either $\bG=0$ or
\begin{equation}
  \Delta t_n < \frac{\delta}{C(\Omega,p)\, \|\bG\|_{L^2}},  \label{eq:nonlocaldtcond}
\end{equation}
where $C(\Omega,p)$ is from the Poincar\'e \eqref{eq:poincare}, then $A_n$ defined by \eqref{eq:nonlocaldefnAn} is coercive and strictly monotone, and thus \eqref{eq:theVI} has a unique solution $u\in\mathcal{K}$.
\end{theorem}

\begin{proof}  Let $\phi=u-v$ for $u,v\in\mathcal{K}$.  Two applications of Cauchy-Schwarz yield
\begin{equation}
\left|\int_\Omega \int_\Omega \bG(x,y) \cdot \grad \phi(x)\,\phi(y) \,dx\,dy\right|
  \le \|\bG\|_{L^2} \|\phi\|^2,   \label{eq:nonlocalGbound}
\end{equation}
By \eqref{eq:nonlocalGbound}, \eqref{eq:nonlocalKpos}, and \eqref{eq:poincare},
\begin{align*}
\ip{A_n(u)-A_n(v)}{\phi} &\ge \|\phi\|_{L^2}^2 - \Delta t_n\,\int_\Omega \int_\Omega \bG(x,y) \cdot \grad \phi(x) \phi(y)\,dx\,dy \\
    &\qquad + \Delta t_n\,\int_\Omega \int_\Omega K(x,y) \grad \phi(x) \cdot \grad \phi(y)\,dx\,dy \\
    &\ge \|\phi\|_{L^2}^2 - \Delta t_n\,\|\bG\|_{L^2} \|\phi\|^2 + \delta \|\grad\phi\|_{L^2}^2 \\
    &\ge \|\phi\|_{L^2}^2 + \left(\frac{\delta}{C(\Omega,p)} - \Delta t_n\,\|\bG\|_{L^2}\right) \|\phi\|^2.
\end{align*}
The result follows from condition \eqref{eq:nonlocaldtcond} and the definition of coercivity. \end{proof}

Theorems \ref{thm:plapadvectwellposed} and \ref{thm:nonlocalwellposed} take different approaches to coercivity.  The former assumes the velocity $\bX$ is differentiable so integration-by-parts gives a time-step criterion based on derivatives of $\bX$.  Theorem \ref{thm:nonlocalwellposed} instead assumes only that $\bG$ is integrable, no integration-by-parts is attempted, and $\Delta t_n$ is bounded using the norm of $\bG$ itself.

\subsection{Explicit time-steps} \label{subsec:explicit}   Suppose $\bq$ is any flux such that, when using the $\theta$-method \eqref{eq:thetamethod} in an implicit case (i.e.~with $\theta>0$), problem \eqref{eq:theVI} is well-posed in $\mathcal{X}=W_0^{1,p}(\Omega)$.  (For example, consider the fluxes in Subsections \ref{subsec:plap} or \ref{subsec:powertransform}.)  Compare the explicit problem, namely a forward Euler step with $\theta=0$, thus $\bQ_n=0$ and
\begin{equation}
F_n = - \Div \bq(\grad u_{n-1},u_{n-1},x) + f(u_{n-1},x).  \label{eq:Fnformexplicit}
\end{equation}
Problem \eqref{eq:theVI} now seeks $u\in\mathcal{K}$ so that
\begin{equation}
\ip{A_n(u)}{\phi} = \int_\Omega (u - \Delta t_n\,F_n - u_{n-1})\phi \ge 0 \quad \text{ for all } \phi \in \mathcal{K}.  \label{eq:explicitweakstep}
\end{equation}
For \eqref{eq:explicitweakstep} to be well-posed the previous state $u_{n-1}$ must be regular enough so that $F_n$ in \eqref{eq:Fnformexplicit} is well defined, that is, $\Div \bq(\grad u_{n-1},u_{n-1},x) \in \mathcal{X}'$ and thus $F_n\in\mathcal{X}'$.  However, even if this holds, VI \eqref{eq:explicitweakstep} is not coercive on $\mathcal{X}=W_0^{1,p}(\Omega)$.

On the other hand, \eqref{eq:explicitweakstep} is well-posed in $\{u_n\ge 0\,\big|\,u_n\in L^2(\Omega)\}$.  The solution is by truncation \cite[page 27]{KinderlehrerStampacchia1980}:
\begin{equation}
u_n = \max\{0,u_{n-1} + \Delta t_n\,F_n\} \in L^2(\Omega). \label{eq:explicittruncation}
\end{equation}
This addresses one time step, but unfortunately $u_n \in L^2(\Omega)$ is not regular enough so that the next timestep has a well-defined weak form.  That is, generally $\Div \bq(\grad u_{n-1},u_{n-1},x)$ need not be in $L^2(\Omega)$.

In summary, for explicit time steps the solution to a single weakly-posed time step is straightforward truncation \eqref{eq:explicittruncation}, but the result is generally not regular enough to yield a well-posed problem at future steps, at least in our discrete-time, continuous-space formulation.  Nonetheless most existing numerical models \cite[for example]{Winkelmannetal2011} proceed by explicit time steps for the fully-discretized problem, followed by truncation where the computed thicknesses come out negative.

\section{Mass conservation and the retreat set}  \label{sec:timeseries}

From now on we assume that the weak problem \eqref{eq:theVI} for a single time-step is well-posed, and that the solutions $u_n\in \mathcal{K}$ are sufficiently-regular so that strong form statements \eqref{eq:semimassconserve} and \eqref{eq:strongconditionwherezero} also hold as described in Theorem \ref{thm:weakimpliesstrong}.  Define
\begin{equation}
M_n = \int_\Omega u_n(x)\,dx \ge 0, \label{eq:totalmassseries}
\end{equation}
the \emph{(total) mass} at time $t_n$.  Recalling set decomposition \eqref{eq:omegadecomposition}, define the \emph{climate input} at time step $n$ as
\begin{equation}
C_n = \Delta t_n\, \int_{\Omega_n} F_n(u_n,x) \label{eq:climateseries}
\end{equation}
Note that we sum values of the source term $F_n(u_n,x)$ only over locations where the fluid is present at $t=t_n$; this is the climate input into the fluid layer.  In the complement $\Omega \setminus \Omega_n = \Omega_n^r \cup \Omega_n^{00}$ the (nonpositive) climate $F_n$ is not removing fluid \emph{at} time $t_n$, though the fluid in $\Omega_n^r$ was completely removed during the time step $[t_{n-1},t_n]$.

Practical models will compute approximations to time-series $M_n$ and $C_n$, or similar, as model outputs, in order to audit mass transfers to and from the fluid layer.  For \emph{fixed}-boundary fluid-layer problems exact discrete mass conservation can be achieved in the sense that
\begin{equation}
M_n = M_{n-1} + C_n \qquad \text{if } \Omega_n = \Omega, \label{eq:oldbalance}
\end{equation}
to within rounding error at each time $t_n$.  For example, if $\Omega_n=\Omega$ then one can easily show \eqref{eq:oldbalance} holds under a Neumann condition $\bQ_n=0$ on $\partial \Omega$ (see below).  However, a balance like \eqref{eq:oldbalance} does \emph{not} follow when there is a nontrivial free boundary such that $\Omega \setminus \Omega_n$ has positive measure.

Let us define the \emph{retreat loss} during the $n$th time step:
\begin{equation}
R_n = \int_{\Omega_n^r} u_{n-1}. \label{eq:retreatlossseries}
\end{equation}
By \eqref{eq:semimassconserve} on $\Omega_n$,
    $$M_n - M_{n-1} = \int_{\Omega_n} (u_n - u_{n-1}) - \int_{\Omega_n^r} u_{n-1} = \Delta t_n \int_{\Omega_n} (- \Div \bQ_n + F_n) \, - R_n.$$
Because $\bQ_n=0$ along $\partial \Omega_n$ by \eqref{eq:Qiscontinuous} and \eqref{eq:Qiszero}, and assuming $\partial \Omega_n$ is Lipschitz,
\begin{equation}
M_n = M_{n-1} + C_n - R_n. \label{eq:newbalance}
\end{equation}

\emph{A posteriori} statement \eqref{eq:newbalance}, replacing \eqref{eq:oldbalance}, suggests what degree of conservation is achievable in time-stepping numerical free-boundary models.  Computing the retreat loss $R_n$ quantifies the conservation error caused by the constraint $u_n\ge 0$.  Consistency suggests $R_n$ should vanish in the $\Delta t_n\to 0$ limit, and in fact the retreat loss $R_n$ can be bounded \emph{a priori} as follows.  Recalling inequality \eqref{eq:strongconditionwherezero}, we have $0 \le u_{n-1} \le -\Delta t_n\,F_n(u_n,x) = -\Delta t_n\,F_n(0,x)$ on $\Omega_n^r$ and thus
\begin{equation}
R_n \le \Delta t_n \int_\Omega \max\{0,-F_n(0,x)\}. \label{eq:retreatbound}
\end{equation}
In words, the retreat loss is bounded by the maximum amount of ablation which the climate can apply to a bare substrate during the time step.  Given a conservation error tolerance, estimate \eqref{eq:retreatbound} can be used to put an upper bound on $\Delta t_n$.

\section{Fully-discrete models}  \label{sec:spacediscretized}

So far we have treated fluid-layer mass conservation models in semi-discretized form, as a sequence of continuous-space free-boundary problems.  We now add spatial discretization, first an unstructured finite volume (FV) method \cite{LeVeque2002}, and later adding a finite element (FE) space of admissible thickness functions, and we reconsider mass conservation in these fully-discretized settings.

\subsection{Unstructured finite volumes} \label{subsec:spacenotation}  To set notation for spatially-discretized schemes, assume $\Omega \subset \RR^d$ is polygonal.  (We will use language suitable for the $\RR^2$ case, ``polygon,'' ``edge,'' and etc.)  Let us assume that $\Omega$ is tiled by open polygonal cells $\omega_j$, indexed by $j\in J$ with $|J|<\infty$, with area $|\omega_j|$, so that $\omega_j \cap \omega_k = \emptyset$ for $j\ne k$, $\bar\Omega = \bigcup_{j\in J} \bar \omega_j$, and $|\Omega| = \sum_{j\in J} |\omega_j|$.  We say that an edge, denoted by the ordered pair $(j,k)$, exists between cell $j$ and cell $k$ if $\bar\omega_j \cap \bar \omega_k$ has positive $(d-1)$-measure (length) $\ell_{(j,k)}>0$.  The set of edges for cell $\omega_j$ is denoted $\mathcal{E}_j=\{k\,\big|\text{edge } (j,k) \text{ exists}\}$.  Note that cells may be non-convex, the number of edges per cell may vary, and hanging nodes are allowed.

Suppose now that the strong form \eqref{eq:semimassconserve} is discretized using the following generic FV scheme.  The discrete thickness $u_n^j$ in cell $j$ is interpreted as an average \cite{LeVeque2002}, and similarly $F_n^j$ denotes the average source term for the cell:
\begin{equation}
u_n^j \approx \frac{1}{|\omega_j|} \int_{\omega_j} u_n(x), \qquad F_n^j \approx \frac{1}{|\omega_j|} \int_{\omega_j} F_n(u_n,x).  \label{eq:fvthickness}
\end{equation}
(One may suppose $F_n^j$ is computed by a quadrature scheme, but such details will not matter.)  The scheme includes some method for calculating discrete (scalar) normal flux across each edge $(j,k)$:
\begin{equation}
Q_n^{(j,k)} \approx \frac{1}{\ell_{(j,k)}} \int_{(j,k)} \bQ_n(\grad u_n,u_n,x) \cdot \bn_{(j,k)}. \label{eq:fvflux}
\end{equation}
Here $\bn_{(j,k)}$ denotes the unit normal vector to edge $(j,k)$ directed outward from $\omega_j$; thus $\bn_{(k,j)} = -\bn_{(j,k)}$.  Presumably the fluxes $Q_n^{(j,k)}$ are approximated using values $\{u_n^l\}$, though again the details are not important.

We now require the scheme to satisfy interior conservation.  That is, we require that between any two adjacent \emph{fluid-filled} cells we have flux balance across the edge:
\begin{equation}
  u_n^j u_n^k > 0 \quad \implies \quad Q_n^{(k,j)}=-Q_n^{(j,k)}.  \label{eq:fvlocalconservation}
\end{equation}
The hypothesis in \eqref{eq:fvlocalconservation} is important.  We do \emph{not} expect discrete conservation at the free boundary, because a flux scheme applied at the edge of a fluid-free (\emph{dry}) cell, facing a fluid-filled (\emph{wet}) cell, cannot be expected to compute a flux which balances the nonzero flux generated by the geometry (and stress state, etc.) of the wet cell.  Indeed, advance of the fluid layer into a dry cell requires flux imbalance at such edges, and likewise for a retreat which leaves behind a dry cell.

Finally we require that if $u_n^j>0$ then the scheme approximates \eqref{eq:semimassconserve} using the obvious FV formula based on the fluxes:
\begin{equation}
\frac{u_n^j - u_{n-1}^j}{\Delta t_n} + \frac{1}{|\omega_j|} \sum_{k\in \mathcal{E}_j} Q_n^{(j,k)} \ell_{(j,k)} = F_n^j. \label{eq:fvmassconserve}
\end{equation}
(Notationally, equation \eqref{eq:fvmassconserve} appears to be the backward Euler scheme, but in fact the time-stepping is quite general; see Section \ref{sec:strongform} and Appendix \ref{app:rk2}.)  However, \eqref{eq:fvmassconserve} only applies when the cell is wet at the end of the time step ($u_n^j>0$).  For dry cells we do not, for now, state any equation other than $u_n^j=0$, but see Subsection \ref{subsec:ncp}.

Many schemes can be given interpretations \eqref{eq:fvthickness}--\eqref{eq:fvmassconserve}, including FV methods for hyperbolic problems \cite{LeVeque2002}, and more-general schemes for diffusive problems \cite{Bueler2016,Morton2018}.  They will differ in how the equations are solved, how the free-boundary conditions are applied, and what are the consequent stability and convergence properties.  Indeed \eqref{eq:fvthickness}--\eqref{eq:fvmassconserve} may not suffice to give a unique scheme even when formulas for the edge fluxes $Q_n^{(j,k)}$ are added, but these axioms suffice to allow the conservation error quantification given next.

\subsection{The discrete-space ``boundary leak''}  \label{subsec:leak}  For schemes satisfying \eqref{eq:fvthickness}--\eqref{eq:fvmassconserve} we now define \emph{a posteriori} computable time series for conservation of mass.  The following discrete formulas, with superscript ``$h$'', have analogs in Section \ref{sec:timeseries}:
\begin{equation}
M_n^h = \sum_j u_n^j |\omega_j|, \quad C_n^h = \Delta t_n\!\!\sum_{u_n^j>0} F_n^j |\omega_j|, \quad R_n^h = \sum_{u_n^j=0} u_{n-1}^j |\omega_j|.  \label{eq:fvtimeseriesdefn}
\end{equation}
Now \eqref{eq:fvmassconserve} implies
\begin{align}
M_n^h - M_{n-1}^h &= \sum_{u_n^j>0} (u_n^j - u_{n-1}^j) |\omega_j| - \sum_{u_n^j=0} u_{n-1}^j |\omega_j| \notag \\
   &= - \Delta t_n\,\sum_{u_n^j>0}\, \sum_{k\in\mathcal{E}_j} Q_n^{(j,k)} \ell_{(j,k)} + C_n^h - R_n^h.  \label{eq:fvnewbalance}
\end{align}
Interior conservation \eqref{eq:fvlocalconservation} reduces the remaining sum to one over edges between wet and dry cells.  We call this residual sum the \emph{boundary leak} (Figure \ref{fig:fvmesh-leak}):
\begin{equation}
B_n^h = \Delta t_n \sum_{u_n^j > 0, u_n^k = 0, k\in\mathcal{E}_j} Q_n^{(j,k)} \ell_{(j,k)}. \label{eq:fvdefineleak}
\end{equation}
This is the net amount of unbalanced flux along the discrete free boundary.

These time series allow us to replace \eqref{eq:newbalance} with a fully-discrete balance:
\begin{equation}
  M_n^h = M_{n-1}^h + C_n^h - R_n^h - B_n^h. \label{eq:fvfinalbalance}
\end{equation}
Note that the masses $M_n^h$ and the retreat losses $R_n^h$ are nonnegative while the climate inputs $C_n^h$ and the boundary leaks $B_n^h$ can be of either sign.

The boundary leak is a numerical error caused by the spatial discretization.  That is, the continuous-space flux along the free-boundary is zero because of the regularity of the solution ($u_n \in W_0^{1,p}(\Omega)$) and by flux conditions \eqref{eq:Qiscontinuous} and \eqref{eq:Qiszero}.  Note that if the free boundary is well-behaved, which is beyond our scope to show even under strong assumptions on the data, and which is nontrivially related to the substrate topography \cite{Bueler2016}, then we expect $B_n^h\to 0$ as $h\to 0$.  Contrast the retreat loss $R_n^h$; it is also a numerical error but it appears in time semi-discretization and it should stabilize at nonzero values under spatial refinement.

\begin{figure}[ht]
\begin{center}
\begin{tikzpicture}[scale=0.35]
  \draw[gray, thick] (3,0) -- (6,3);
  \draw[gray, thick] (0,5) -- (3,5);
  \draw[gray, thick] (0,10) -- (4,9);
  \draw[gray, thick] (0,15) -- (4,13);
  \draw[gray, thick] (4,13) -- (4,9);
  \draw[gray, thick] (4,13) -- (8,14);
  \draw[gray, thick] (8,14) -- (8,15);
  \draw[gray, thick] (12,11) -- (18,15);
  \draw[gray, thick] (10,5) -- (14,3);
  \draw[gray, thick] (10,1) -- (14,3);
  \draw[gray, thick] (14,3) -- (18,3);
  \draw[gray, thick] (6,3) -- (10,1);
  \draw[gray, thick] (6,3) -- (10,5);
  \draw[gray, thick] (10,0) -- (10,1);
  \draw[gray, thick] (10,5) -- (8,8);
  \draw[gray, thick] (8,8) -- (12,11);
  \draw[gray, thick] (16,9) -- (18,9);

  \draw[line width=2.0pt] (6,3) -- (3,5) -- (4,9) -- (8,8) -- (8,14) --
                          (12,11) -- (16,9) -- (10,5) -- (14,3) -- (10,1) -- cycle;
  \draw (6,6) node {$u>0$};
  \draw (11,8) node {$u>0$};
  \draw (9.7,10.8) node {$u>0$};
  \draw (10,3) node {$u>0$};
  \draw (2,2) node {$0$};
  \draw (1.5,7.5) node {$0$};
  \draw (2,12) node {$0$};
  \draw (4.5,14.5) node {$0$};
  \draw (6,11) node {$0$};
  \draw (12,14) node {$0$};
  \draw (15.5,11.5) node {$0$};
  \draw (13.5,5.5) node {$0$};
  \draw (13,0.5) node {$0$};
  \draw (6.5,0.5) node {$0$};
\end{tikzpicture}
\end{center}
\caption{The ``boundary leak'' $B_n^h$ is computed along those edges where wet and dry cells meet.}
\label{fig:fvmesh-leak}
\end{figure}
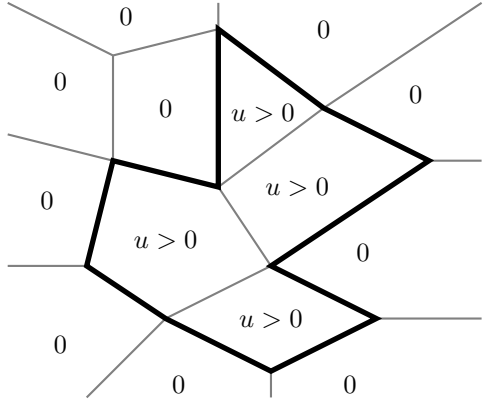

In summary, in a free-boundary FV computation we may report computable time series $\{M_n^h,C_n^h,R_n^h,B_n^h\}$ which balance as in \eqref{eq:fvfinalbalance} (up to rounding error).  Time series $R_n^h$ and $B_n^h$ are conservation errors intrinsic to free-boundary fluid layer models.

\subsection{Complementarity and cell-wise conservation}  \label{subsec:ncp}  The continuous-space, discrete-time weak formulation in Sections \ref{sec:weakform} and \ref{sec:wellposed}, using variational inequalities (VIs) \eqref{eq:theVI}, would often be solved using finite element (FE) discretization \cite[for example]{CalvoDuranyVazquez2000,JouvetBueler2012,
JouvetBuelerGraeserKornhuber2013}, but we have just applied the FV language of discrete conservation.  These views can be harmonized by observing that a VI is equivalent to a nonlinear complementarity problem (NCP) \cite{FacchineiPang2003,KinderlehrerStampacchia1980}, and both practical solver algorithms and clearer intuition result from this observation.  The dual-mesh schemes described next are both conforming and implementable using finite-dimensional NCP solvers.  The shallow ice sheet solver described in \cite{Bueler2016} serves as an example of the combined FV, FE, and VI/NCP techniques described here.

Suppose we discretize using an FE subspace $S^h \subset \mathcal{X} = W_0^{1,p}(\Omega)$, with a nodal basis of $m$ hat functions $\{\psi_i\}$, based on a triangulation (or other mesh) of $\Omega$ with resolution $h$.  Consider problem \eqref{eq:theVI} on this space, namely
\begin{equation}
\ip{A_n(u_n^h)}{v^h-u_n^h} \ge 0 \quad \text{for all } v \in \mathcal{K} \cap S^h,   \label{eq:FEtheVI}
\end{equation}
where $A_n$ is given by \eqref{eq:defineAn} and (as usual) $\mathcal{K} = \{u\in \mathcal{X}\,\big|\,u\ge 0\}$.  Under the same Section \ref{sec:wellposed} hypotheses considered for \eqref{eq:theVI}, we assume problem \eqref{eq:FEtheVI} is well-posed for $u_n^h \in \mathcal{K} \cap S^h$.  Next we suppose the nodal basis is admissible so that $\psi_i(x)\ge 0$ on $\Omega$ and $v(x) = \sum_{i=1}^m v(x_i) \psi_i(x)$ for the nodes $x_i$.  (For example, the usual hat-function bases for $P_1$ and $Q_1$ elements would satisfy this hypothesis, but not the nodal $P_2$ basis \cite{Elmanetal2014}.)  Then we can represent the FE solution $u_n^h$ by a vector $\tilde u \in \RR_+^m$, i.e.~$\tilde u_i = u_n^h(x_i) \ge 0$.

Up to isomorphism the nonlinear operator in FE formulation \eqref{eq:FEtheVI} is a map $\tilde A:\RR_+^m \to \RR^m$ with entries $\tilde A(\tilde u)^i = \ip{A_n(u_n^h)}{\psi_i} \in \RR$.  The finite-dimensional VI \eqref{eq:FEtheVI} is equivalent to the nonlinear complementarity problem (NCP)
\begin{equation}
\tilde u_i \ge 0, \quad \tilde A(\tilde u)^i \ge 0, \quad \tilde u_i \tilde A(\tilde u)^i = 0 \label{eq:FEtheNCP}
\end{equation}
\cite[Theorem I.5.5]{KinderlehrerStampacchia1980}; see also \cite{FacchineiPang2003}.  (By nonnegativity of the factors the complementarity condition can be regarded either entrywise or as an inner-product.)  NCP \eqref{eq:FEtheNCP} is nonlinear even if the operator $A_n$ is linear, and thus iteration is expected in any numerical solution.  Scalable Newton schemes for NCP problems are described in \cite{BensonMunson2006}; relevant applications appear in \cite{Brinkerhoffetal2017,Bueler2016}.

In our fluid-layer context the intuition behind NCP \eqref{eq:FEtheNCP} is straightforward.  Namely, at convergence of the numerical solver:
\begin{itemize}[leftmargin=25mm]
\item[$\tilde u_i\ge 0$\,:] the layer thickness at each node $x_i$ is nonnegative,
\item[$\tilde A(\tilde u)^i \ge 0$\,:] the balance between flow and climate inputs, represented by the residual of the operator $A_n$ in the direction of test function $\psi_i$, never removes more mass than was already present, and
\item[$\tilde u_i \tilde A(\tilde u)^i = 0$\,:] at each location $x_i$ either the thickness is zero or the flow and climate are in exact balance.
\end{itemize}

\smallskip
When a value $\tilde A(\tilde u)^i$ is zero then mass conservation (balance) equation \eqref{eq:semimassconserve} holds at node $x_i$, but only in an FE sense.  That is, a weighted-average of the integrand in \eqref{eq:defineAn}, over the support of $\psi_i(x)$, is zero.  Tradition and climate-modeling practice regards such an averaged sense of discrete balance as inferior to exact local balance \eqref{eq:fvlocalconservation}.  However, we may adapt \eqref{eq:FEtheNCP} to an FV view by assuming that for each FE node $x_i$ there is a unique corresponding FV cell $\omega_i$ (Subsection \ref{subsec:spacenotation}).  Schemes satisfying this condition, such as \cite{Bueler2016,EwingLinLin2002,Ringleretal2013}, have a ``dual mesh,'' namely cells for conservation plus a mesh for representing the solution.  Note we will need no detailed assumptions about the mesh correspondence in the following computations.

Now we compute the residual for the characteristic function $\mathbbm{1}_{\omega_i}$ of an FV cell:
\begin{align}
\hat A(\tilde u)^i &= \ip{A_n(u_n^h)}{\mathbbm{1}_{\omega_i}}  \label{eq:cellwiseresidual} \\
                   &= \int_{\omega_i} \left(u_n^h - \Delta t_n F_n^h - u_{n-1}^h\right) + \Delta t_n \sum_{k\in\mathcal{E}_i} \int_{(i,k)} \bQ_n^h\cdot \bn_{(i,k)}. \notag
\end{align}
where $F_n^h(x) = F_n(u_n^h,x)$, $\bQ_n^h(x) = \bQ_n(\grad u_n^h, u_n^h,x)$, and $i\in\{1,\dots,m\}$.  Regarding the flux integral on the right we again assume interior balance \eqref{eq:fvlocalconservation}.  The integral $\ip{A_n(u_n^h)}{\mathbbm{1}_{\omega_i}}$ must be understood in a distributional sense, for instance as a limit using mollification of $\mathbbm{1}_{\omega_i}$.

The NCP corresponding to the VI for \eqref{eq:cellwiseresidual}, namely
\begin{equation}
\tilde u_i \ge 0, \quad \hat A(\tilde u)^i \ge 0, \quad \tilde u_i \hat A(\tilde u)^i = 0 \label{eq:FVtheNCP}
\end{equation}
in $\RR^m$, has an interpretation as before even though it mixes FE and FV aspects.  For each cell $\omega_i$ the nodal thickness $\tilde u_i$ is nonnegative, the flow and climate will not remove more mass than was already present in the cell ($\hat A(\tilde u)^i \ge 0$), and either the nodal thickness is zero or conservation (balance) is exact in a cell-wise sense.  Note that $\mathbbm{1}_{\omega_i} \notin \mathcal{X}$, so this is a Petrov-Galerkin formulation, but the scheme is conforming in the sense that $u_n^h\in \mathcal{K}\cap S^h$ is admissible \cite{Elmanetal2014}.  Such a combined ``finite volume element'' viewpoint is not new as it applies to PDE problems \cite[for example]{Cai1990,EwingLinLin2002}, but it seems not to have been used for VIs except in \cite{Bueler2016}.

However, solution of \eqref{eq:FVtheNCP} implies revised mass accounting relative to Subsection \ref{subsec:leak}.  We redefine
\begin{equation}
\hat M_n^h = \int_\Omega u_n^h, \quad \hat C_n^h = \Delta t_n \sum_{\tilde u_i>0} \int_{\omega_i} F_n^h, \quad \hat R_n^h = \sum_{\tilde u_i=0} \int_{\omega_i} u_{n-1}^h, \label{eq:newfvtimeseriesdefn}
\end{equation}
to replace \eqref{eq:fvtimeseriesdefn}, and
\begin{equation}
\hat B_n^h = \Delta t_n \sum_{\tilde u_i > 0, \tilde u_k = 0, k\in\mathcal{E}_j} \int_{(i,k)} \bQ_n^h \cdot \bn_{(i,k)} \label{eq:newfvdefineleak}
\end{equation}
to replace \eqref{eq:fvdefineleak}.  Noting that $u_n^h$ may be nonzero on a cell $\omega_i$ corresponding to a zero nodal thickness $\tilde u_i=0$, the following calculation applies if $\tilde u$ solves NCP \eqref{eq:FVtheNCP}:
\begin{align}
\hat M_n^h - \hat M_{n-1}^h &= \sum_{\tilde u_i>0} \int_{\omega_i} u_n^h - u_{n-1}^h + \sum_{\tilde u_i=0} \int_{\omega_i} u_n^h - u_{n-1}^h \label{eq:newfvcalculation} \\
  &= \hat C_n^h - \Delta t_n \sum_{\tilde u_i>0} \sum_{k\in \mathcal{E}_i} \int_{(i,k)} \bQ_n^h \cdot \bn_{(i,k)} + \sum_{\tilde u_i=0} \int_{\omega_i} u_n^h - \hat R_n^h \notag
\end{align}
The flux sum again simplifies through cancellation by interior conservation \eqref{eq:fvlocalconservation}, but now we must add a new time series, which we call the \emph{cell slop}, because the support of $u_n^h$ generally extends outside of the wet cells:
\begin{equation}
\hat S_n^h = \sum_{\tilde u_i=0} \int_{\omega_i} u_n^h. \label{eq:fvdefineslop}
\end{equation}

With the revised definitions, by \eqref{eq:newfvcalculation} the following balance holds,
\begin{equation}
\hat M_n^h = \hat M_{n-1}^h + \hat C_n^h - \hat R_n^h - \hat B_n^h + \hat S_n^h, \label{eq:newfvfinalbalance}
\end{equation}
now replacing both \eqref{eq:newbalance} and \eqref{eq:fvfinalbalance}.  Time series \eqref{eq:newfvtimeseriesdefn}, \eqref{eq:newfvdefineleak}, and \eqref{eq:fvdefineslop} are computable \emph{a posteriori} although quadrature may be needed depending on the form of functions $F_n$ and $\bQ_n$.

To summarize, \eqref{eq:newfvfinalbalance} identifies three conservation errors for free-boundary problems which are not present in the fixed-boundary case.  The retreat loss $R_n^h$ goes to zero under temporal refinement (Section \ref{sec:timeseries}), the boundary leak $B_n^h$ goes to zero under spatial refinement (Subsection \ref{subsec:leak}), and the cell slop $\hat S_n^h$ is identically zero in a pure FV formulation.

\section{Conclusion} \label{sec:conclusion}

Global-scale fluid models sometimes claim exact discrete conservation as a goal \cite{Ringleretal2013,Thuburn2008}, but these claims are apparently made in a fixed-boundary context, while climate models increasingly incorporate free-boundary submodels.  Such multiphysics Earth system models need to conserve masses of the phases of water separately as they have different physical properties relevant to climate dynamics.  (For example, snow and ice have higher albedo and lower density than the liquid ocean.)  Within such models it is common for one or more fluids or phases to form a thin layer with a moving (free) lateral boundary, a description which applies to ice sheets, glaciers, ice shelves, sub-glacial liquid water, sea ice, and evaporable seas and lakes, among others.  Existing models sometimes include \emph{ad hoc} redistribution schemes, which globally balance the mass-conservation books, but we assert that discrete mass conservation cannot otherwise occur in such free-boundary subsystems, though conservation is recoverable in the temporal and spatial refinement limit.  Conscientious numerical model design therefore suggests quantification of conservation errors, not sweeping them under the refinement-limit (or other) rugs.

We have addressed the modeling of thin fluid layers through semidiscretization in time (Section \ref{sec:strongform}), and then weak formulation as a sequence of continuous-space VIs (Sections \ref{sec:weakform}--\ref{sec:timeseries}), always based on the fundamental nonnegative thickness condition.  On the other hand, spatial discretization must also be applied, so we interpret discrete mass conservation errors first through an FV framework (Section \ref{sec:spacediscretized}), then reconciling this viewpoint to FE solution of the VIs (Subsection \ref{subsec:ncp}).  The essential intent of Section \ref{sec:spacediscretized} is, in fact, to recommend that modelers do conservation arithmetic on the finite-dimensional NCP or VI form of the problem solved at each time step.

For numerical models we have identified the per time-step retreat set $\Omega_n^r$ (Subsection \ref{subsec:setdecompose}) and retreat mass loss $R_n$ (Section \ref{sec:timeseries}) as most fundamental.  Here $\Omega_n^r$ is the (continuous-space) region where the fluid layer thickness is positive at the beginning of the time step, and, through flow and (climatic) source terms, becomes zero at the end of the step.  By definition, fluid is completely removed from the retreat set at some time during the time step, and, intuitively, the numerical model has no access to the (substep) time and manner in which this occurs, other than in the inequality sense that the climate was sufficiently ablative so as to eliminate that fluid.  Note that the retreat area $|\Omega_n^r|$ can be arbitrarily large even for short time steps.  For example, in an ablating climate a large area of thin ice sheet or sea ice can melt, or a large area of water can evaporate and expose bare ground, in a short time.  The retreat loss $R_n$, a mass, can be bounded \emph{a priori} (Section \ref{sec:timeseries}), but still it cannot be exactly-balanced by a computable integral of the climatic source term during the time step.

These conclusions about retreat, which apply in the semi-discretized and continuous-space case, are independent of any particular spatial discretization scheme.  However, in Section \ref{sec:spacediscretized} we define conservation error quantities at the discretized free boundary.  With these computable time series in hand a numerical model can balance the books up to rounding error in a manner which properly reflects the free-boundary character of the model.  Even without \emph{a priori} control of the free boundary, a user can assess whether \emph{a posteriori} conservation errors are acceptably small, and shorten time steps or refine meshes if not.  Climate models, in particular, can thereby control some of the uncertainty in mass transfers between component fluids of the Earth system.

\bibliographystyle{siamplain}
\bibliography{lc}

\begin{thebibliography}{10}

\bibitem{Albrechtetal2011}
{\sc T.~Albrecht, M.~Martin, M.~Haseloff, R.~Winkelmann, and A.~Levermann},
  {\em Parameterization for subgrid-scale motion of ice-shelf calving fronts},
  The Cryosphere, 5 (2011), pp.~35--44.

\bibitem{Alexander1977}
{\sc R.~Alexander}, {\em Diagonally implicit {R}unge-{K}utta methods for stiff
  {ODE}s}, SIAM J. Numer. Anal., 14 (1977), pp.~1006--1021.

\bibitem{AlonsoSantillanaDawson2008}
{\sc R.~Alonso, M.~Santillana, and C.~Dawson}, {\em On the diffusive wave
  approximation of the shallow water equations}, Eur. J. Appl. Math., 19
  (2008), pp.~575--606.

\bibitem{AscherPetzold1998}
{\sc U.~Ascher and L.~Petzold}, {\em Computer {M}ethods for {O}rdinary
  {D}ifferential {E}quations and {D}ifferential-algebraic {E}quations}, SIAM
  Press, Philadelphia, PA, 1998.

\bibitem{Aschwandenetal2012}
{\sc A.~Aschwanden, E.~Bueler, C.~Khroulev, and H.~Blatter}, {\em An enthalpy
  formulation for glaciers and ice sheets}, J. Glaciol., 58 (2012),
  pp.~441--457.

\bibitem{BarrettLiu1993}
{\sc J.~W. Barrett and W.~B. Liu}, {\em Finite element approximation of the
  $p$-{L}aplacian}, Math. Comp., 61 (1993), pp.~523--537.

\bibitem{BensonMunson2006}
{\sc S.~Benson and T.~Munson}, {\em Flexible complementarity solvers for
  large-scale applications}, Optimization Methods and Software, 21 (2006),
  pp.~155--168.

\bibitem{Brinkerhoffetal2017}
{\sc D.~Brinkerhoff, M.~Truffer, and A.~Aschwanden}, {\em Sediment transport
  drives tidewater glacier periodicity}, Nature Commun., 8 (2017).

\bibitem{Bueler2016}
{\sc E.~Bueler}, {\em Stable finite volume element schemes for the shallow ice
  approximation}, J. Glaciol., 62 (2016), pp.~230--242.

\bibitem{BuelervanPelt2015}
{\sc E.~Bueler and W.~van Pelt}, {\em Mass-conserving subglacial hydrology in
  the {Parallel Ice Sheet Model} version 0.6}, Geoscientific Model Development,
  8 (2015), pp.~1613--1635.

\bibitem{Cai1990}
{\sc Z.~Cai}, {\em On the finite volume element method}, Numerische Mathematik,
  58 (1990), pp.~713--735.

\bibitem{CalvoDuranyVazquez2000}
{\sc N.~Calvo, J.~Durany, and C.~V\'azquez}, {\em Numerical computation of ice
  sheet profiles with free boundary models}, Appl. Numer. Math., 35 (2000),
  pp.~111--128.

\bibitem{DiazSchiavi1999}
{\sc J.~Diaz and E.~Schiavi}, {\em On a degenerate parabolic/hyperbolic system
  in glaciology giving rise to a free boundary}, Nonlinear Analysis, Theory,
  Methods \& Applications, 38 (1999), pp.~649--673.

\bibitem{EgholmNielsen2010}
{\sc D.~Egholm and S.~Nielsen}, {\em An adaptive finite volume solver for ice
  sheets and glaciers}, J. Geophys. Res.: Earth Surface, 115 (2010).

\bibitem{Elmanetal2014}
{\sc H.~C. Elman, D.~J. Silvester, and A.~J. Wathen}, {\em Finite {E}lements
  and {F}ast {I}terative {S}olvers: with {A}pplications in {I}ncompressible
  {F}luid {D}ynamics}, Oxford University Press, 2nd~ed., 2014.

\bibitem{Evans2010}
{\sc L.~Evans}, {\em Partial {D}ifferential {E}quations}, American Mathematical
  Society, 2nd~ed., 2010.

\bibitem{EwingLinLin2002}
{\sc R.~E. Ewing, T.~Lin, and Y.~Lin}, {\em On the accuracy of the finite
  volume element method based on piecewise linear polynomials}, SIAM J. Numer.
  Analysis, 39 (2002), pp.~1865--1888.

\bibitem{FacchineiPang2003}
{\sc F.~Facchinei and J.-S. Pang}, {\em Finite-{D}imensional {V}ariational
  {I}nequalities and {C}omplementarity {P}roblems}, vol.~1, Springer, 2003.

\bibitem{Friedman1982}
{\sc A.~Friedman}, {\em Variational {I}nequalities and {F}ree {B}oundary
  {P}roblems}, Wiley Interscience, 1982.

\bibitem{GeorgeIverson2014}
{\sc D.~L. George and R.~M. Iverson}, {\em A depth-averaged debris-flow model
  that includes the effects of evolving dilatancy. {II}. {N}umerical
  predictions and experimental tests}, P. Roy. Soc. A-Math. Phy., 470 (2014).

\bibitem{GilbargTrudinger2001}
{\sc D.~Gilbarg and N.~Trudinger}, {\em Elliptic {P}artial {D}ifferential
  {E}quations of {S}econd {O}rder}, Springer, 2001.
\newblock Reprint of the 1998 edition.

\bibitem{GlowinskiMarroco1975}
{\sc R.~Glowinski and A.~Marroco}, {\em Sur l'approximation, par \'el\'ements
  finis d'ordre un, et la r\'esolution, par p\'enalisation-dualit\'e d'une
  classe de probl\`emes de {D}irichlet non lin\'eaires}, {R.A.I.R.O.}, 9
  (1975), pp.~41--76.

\bibitem{IdelsohnOnate2010}
{\sc S.~R. Idelsohn and E.~O{\~n}ate}, {\em The challenge of mass conservation
  in the solution of free-surface flows with the fractional-step method:
  Problems and solutions}, International Journal for Numerical Methods in
  Biomedical Engineering, 26 (2010), pp.~1313--1330.

\bibitem{JaroschSchoofAnslow2013}
{\sc A.~H. Jarosch, C.~G. Schoof, and F.~S. Anslow}, {\em Restoring mass
  conservation to shallow ice flow models over complex terrain}, The
  Cryosphere, 7 (2013), pp.~229--240.

\bibitem{JouvetBueler2012}
{\sc G.~Jouvet and E.~Bueler}, {\em Steady, shallow ice sheets as obstacle
  problems: well-posedness and finite element approximation}, SIAM J. Appl.
  Math., 72 (2012), pp.~1292--1314.

\bibitem{JouvetBuelerGraeserKornhuber2013}
{\sc G.~Jouvet, E.~Bueler, C.~Gräser, and R.~Kornhuber}, {\em A nonsmooth
  {N}ewton multigrid method for a hybrid, shallow model of marine ice sheets},
  in Recent Advances in Scientific Computing and Applications, vol.~586 of
  Contemporary Mathematics, American Mathematical Society, 2013, pp.~197--205.

\bibitem{Jouvetetal2011}
{\sc G.~Jouvet, J.~Rappaz, E.~Bueler, and H.~Blatter}, {\em Existence and
  stability of steady state solutions of the shallow ice sheet equation by an
  energy minimization approach}, J. Glaciol., 57 (2011), pp.~345--354.

\bibitem{KinderlehrerStampacchia1980}
{\sc D.~Kinderlehrer and G.~Stampacchia}, {\em An {I}ntroduction to
  {V}ariational {I}nequalities and their {A}pplications}, Pure and Applied
  Mathematics, Academic Press, 1980.

\bibitem{Kondic2003}
{\sc L.~Kondic}, {\em Instabilities in gravity driven flow of thin fluid
  films}, SIAM Rev., 45 (2003), pp.~95--115 (electronic).

\bibitem{LeVeque2002}
{\sc R.~J. LeVeque}, {\em Finite Volume Methods for Hyperbolic Problems},
  Cambridge Texts in Applied Mathematics, Cambridge University Press, 2002.

\bibitem{LeVequeetal2011}
{\sc R.~J. LeVeque, D.~L. George, and M.~J. Berger}, {\em Tsunami modelling
  with adaptively refined finite volume methods}, Acta Numerica, 20 (2011),
  pp.~211--289.

\bibitem{LipscombHunke2004}
{\sc W.~H. Lipscomb and E.~C. Hunke}, {\em Modeling sea ice transport using
  incremental remapping}, Mon. Wea. Rev., 132 (2004), pp.~1341--1354.

\bibitem{Maxwelletal2015}
{\sc R.~M. Maxwell, L.~E. Condon, and S.~J. Kollet}, {\em A high-resolution
  simulation of groundwater and surface water over most of the continental {US}
  with the integrated hydrologic model {ParFlow} v3}, Geoscientific Model
  Development, 8 (2015), pp.~923--937.

\bibitem{Morton2018}
{\sc K.~W. Morton}, {\em Numerical {S}olution of {C}onvection-{D}iffusion
  {P}roblems}, CRC Press, 2018.
\newblock reprint of the 1996 edition.

\bibitem{MortonMayers2005}
{\sc K.~W. Morton and D.~F. Mayers}, {\em Numerical {S}olutions of {P}artial
  {D}ifferential {E}quations: {A}n {I}ntroduction}, Cambridge University Press,
  2nd~ed., 2005.

\bibitem{Peral1997}
{\sc I.~Peral}, {\em Multiplicity of solutions for the $p$-{L}aplacian}.
\newblock Notes of the Second International School in Functional Analysis and
  Applications to Differential Equations, ICTP-Trieste, 1997.

\bibitem{PorterStirling1990}
{\sc D.~Porter and D.~Stirling}, {\em Integral {E}quations: {A} {P}ractical
  {T}reatment, from {S}pectral {T}heory to {A}pplications}, Cambridge
  University Press, 1990.

\bibitem{Raviart1970}
{\sc P.~A. Raviart}, {\em Sur la r\'esolution de certaines equations
  paraboliques non lin\'eaires}, J. Functional Anal., 5 (1970), pp.~299--328.

\bibitem{Ringleretal2013}
{\sc T.~Ringler, M.~Petersen, R.~Higdon, D.~Jacobsen, P.~Jones, and
  M.~Maltrud}, {\em A multi-resolution approach to global ocean modeling},
  Ocean Modelling, 69 (2013), pp.~211--232.

\bibitem{Schoofetal2012}
{\sc C.~Schoof, I.~J. Hewitt, and M.~A. Werder}, {\em Flotation and free
  surface flow in a model for subglacial drainage. {P}art {I}: {D}istributed
  drainage}, J. Fluid Mech., 702 (2012), pp.~126--156.

\bibitem{Thorndikeetal1975}
{\sc A.~Thorndike, D.~Rothrock, G.~Maykut, and R.~Colony}, {\em The thickness
  distribution of sea ice}, J. Geophys. Res., 80 (1975), pp.~4501--4513.

\bibitem{Thuburn2008}
{\sc J.~Thuburn}, {\em Some conservation issues for the dynamical cores of
  {NWP} and climate models}, J. Comput. Phys., 227 (2008), pp.~3715--3730.

\bibitem{Vazquez2007}
{\sc J.~L. V{\'a}zquez}, {\em The {P}orous {M}edium {E}quation}, Oxford
  University Press, 2007.

\bibitem{Winkelmannetal2011}
{\sc R.~Winkelmann, M.~A. Martin, M.~Haseloff, T.~Albrecht, E.~Bueler,
  C.~Khroulev, and A.~Levermann}, {\em The {P}otsdam {P}arallel {I}ce {S}heet
  {M}odel ({PISM-PIK}) {P}art 1: {M}odel description}, The Cryosphere, 5
  (2011), pp.~715--726.

\end{thebibliography}

\appendix

\section{Inequalities for $p$-norms}   \label{app:pinequalities}  Versions of the inequalities in the next two Lemmas appear in the literature, at least as early as \cite{GlowinskiMarroco1975}, but here the results apply in $\RR^d$---contrast \cite{BarrettLiu1993} for the $\RR^2$ case---and have complete proofs and explicit constants.  The first two proofs follow \cite[Appendix A]{Peral1997}.

\begin{lemma}  \label{lem:pbiginequality}  If $p\ge 2$ and $x,y\in\RR^d$ then
\begin{equation}
\left(|x|^{p-2} x - |y|^{p-2} y\right)\cdot(x-y) \ge 2^{2-p} |x-y|^p. \label{eq:pbiginequality}
\end{equation}
The constant is sharp; consider $y=-x$.
\end{lemma}

\begin{proof}  The case where $x=0$ or $y=0$ is trivial, so assume, by swapping $x$ and $y$ as necessary, that $0 < |y| \le |x|$.  Define $t=|y|/|x|$ and $s = (x\cdot y)/(|x||y|)$ so that $0 < t \le 1$ and $|s|\le 1$.  Expand \eqref{eq:pbiginequality} and divide it by $|x|^p$, to get the equivalent statement
    $$1 - (t^{p-1}+t) s + t^p \ge 2^{2-p} \left(1 - 2 s t + t^2\right)^{p/2}.$$
It is easy to check that this holds when $s=1$, so now we will prove that $2^{2-p}$ is a lower bound for
	$$f(t,s) = \frac{1 - (t^{p-1}+t) s + t^p}{\left(1 - 2 s t + t^2\right)^{p/2}}.$$
on $(t,s) \in R=[0,1]\times[-1,1)$.  Note $1-2st+t^2 > 0$ on $ R$, so $f(t,s)$ is well-defined and differentiable on $R$.

Now, $f(t,-1) = \left(1 + t^{p-1}\right) / \left(1 + t\right)^{p-1}$ on $t\in[0,1]$.  Because $h(t)=t^{p-1}$ is convex for $p \ge 2$,
    $$\frac{1}{2^{p-1}} (1+t)^{p-1} = h(\tfrac{1}{2} 1 + \tfrac{1}{2} t) \le \tfrac{1}{2} h(1) + \tfrac{1}{2} h(t) = \tfrac{1}{2} (1 + t^{p-1}),$$
and thus $f(t,-1) \ge 2^{2-p}$.  On the other hand, a quick calculation shows
    $$\frac{\partial f}{\partial s} = \frac{t}{\left(1 - 2 s t + t^2\right)^{(p+2)/2}} g(t,s)$$
where
    $$g(t,s) = s(2-p) t (t^{p-2} + 1) + (p-1) (t^p+1) - t^{p-2} - t^2$$
is continuous on the closed rectangle $\bar R = [0,1]\times[-1,1]$.  We will show $g(t,s)\ge 0$ on $\bar R$, thus that $\partial f/\partial s \ge 0$ on $R$, and thus that $f(t,s)\ge f(t,-1) \ge  2^{2-p}$ on $R$.

Now,
    $$\frac{\partial g}{\partial s} = (2-p) t (t^{p-2} + 1) \le 0$$
on $\bar R$.  Define $G(t) = g(t,1)$.  We will show $G(t)\ge 0$ on $[0,1]$, thus that $g(t,s)\ge g(t,1)\ge 0$ on $\bar R$.  But $G(t)\ge 0$ is equivalent to $(p-1) (t-1) (t^{p-1}-1) \ge (t^{p-2} - t) (1 - t)$ which is in turn equivalent to $(p-1) (1 - t^{p-1}) \ge t^{p-2} - t$.  Note $(p-1) (1 - t^{p-1}) \ge 0$.  If $p\ge 3$ then $t^{p-2} - t \le 0$ so $G(t)\ge 0$ in that case.  On the other hand, if $2\le p < 3$ then
	$$\frac{t^{p-2} - t}{1 - t^{p-1}} = t^{p-2} \frac{1 - t^{3-p}}{1 - t^{p-1}} \le t^{p-2} \le 1 \le p-1$$
on $t\in[0,1)$, because $t^{p-1}\le t^{3-p}$ and thus $1 - t^{p-1} \ge 1 - t^{3-p}$.  But also $G(1)=0$, so $G(t)\ge 0$ on $[0,1]$. \end{proof}

\begin{lemma}  \label{lem:psmallinequality}  If $1<p\le 2$ and $x,y\in\RR^n$ then
\begin{equation}
\left(|x|^{p-2} x - |y|^{p-2} y\right)\cdot(x-y) \ge (p-1)\, |x-y|^2 \, \left(|x|+|y|\right)^{p-2}. \label{eq:psmallinequality}
\end{equation}
\end{lemma}

\begin{proof}  Assuming $x,y$ are not both zero, by symmetry (swapping $x$ and $y$) and homogeneity (replacing $x,y$ with $\lambda x,\lambda y$) we can assume $|x| = 1 \ge |y|$.  Furthermore, by choosing a basis of $\RR^d$ we can have $x=(1,0,\dots,0)$ and $y=(y_1,y_2,0,\dots,0)$ where $y_1^2+y_2^2 \le 1$.  In these terms, the inequality we seek to prove is
\begin{align*}
&\left(1 - (y_1^2+y_2^2)^{\frac{p-2}{2}} y_1\right) (1-y_1) + (y_1^2+y_2^2)^{\frac{p-2}{2}} y_2^2 \\
&\qquad\qquad \ge (p-1)\, \left((1-y_1)^2+y_2^2\right) \left(1 + \sqrt{y_1^2+y_2^2} \right)^{p-2}.
\end{align*}
(Compare equation (A.4) in \cite{Peral1997}.)  But
\begin{align*}
1 - (y_1^2+y_2^2)^{\frac{p-2}{2}} y_1
      &\ge \begin{cases} 1-y_1, & y_1 \le 0, \\
                        1-y_1^{p-1}, & 0 \le y_1 \le 1 \end{cases}\Bigg\}
      \ge (p-1) (1-y_1).
\end{align*}
(The lower case in the last inequality is easy to prove by the mean-value-theorem applied to $\varphi(t)=t^{p-1}$, for which $\varphi'(1)=p-1$ is the minimum value of the derivative on $t\in[0,1]$.)  Also noting $(y_1^2+y_2^2)^{\frac{p-2}{2}} \ge 1$ and $\left(1 + \sqrt{y_1^2+y_2^2} \right)^{2-p} \ge 1$, because $|y|\le 1$ and $p-2\le 0$, thus
\begin{align*}
&\frac{\left(1 - (y_1^2+y_2^2)^{\frac{p-2}{2}} y_1\right) (1-y_1) + (y_1^2+y_2^2)^{\frac{p-2}{2}} y_2^2}
      {\left((1-y_1)^2+y_2^2\right) \left(1 + \sqrt{y_1^2+y_2^2} \right)^{p-2}} \\
&\qquad \ge \frac{(p-1) (1-y_1)^2 + y_2^2}
      {(1-y_1)^2+y_2^2} \,  \left(1 + \sqrt{y_1^2+y_2^2} \right)^{2-p} \\
&\qquad \ge \frac{(p-1) (1-y_1)^2 + (p-1) y_2^2}{(1-y_1)^2+y_2^2} = p-1.
\end{align*}
This proves \eqref{eq:psmallinequality}. \end{proof}

We will also need the following result of combining point-wise Lemma \ref{lem:psmallinequality} with integration over a set $\Omega$.

\begin{lemma} \label{lem:smallpbound}  Suppose $1<p\le 2$.  If $\Omega \subset \RR^d$ is measurable and if $\bu,\bv\in L^p(\Omega; \RR^k)$ for $k\ge 1$, then
\begin{equation}
    \int_\Omega \frac{|\bu-\bv|^p}{\left(|\bu|+|\bv|\right)^{2-p}} \ge \frac{\|\bu-\bv\|_{L^p}^2}{\big\||\bu|+|\bv|\big\|_{L^p}^{2-p}}. \label{eq:smallpbound}
\end{equation}
\end{lemma}

\begin{proof}  By H\"older inequality with $r=2/p$ and $s=2/(2-p)$, so $r^{-1}+s^{-1}=1$,
\begin{align*}
\int_\Omega |\bu - \bv|^p &= \int_\Omega \frac{|\bu-\bv|^p}{\left(|\bu|+|\bv|\right)^{p(2-p)/2}} \left(|\bu|+|\bv|\right)^{p(2-p)/2} \\
    &\le \left(\int_\Omega \frac{|\bu-\bv|^2}{\left(|\bu|+|\bv|\right)^{2-p}}\right)^{p/2} \left(\int_\Omega \left(|\bu|+|\bv|\right)^p\right)^{(2-p)/2},
\end{align*}
thus \eqref{eq:smallpbound}.
\end{proof}

Finally we recall the Poincar\'e inequality on the Sobolev space $W_0^{1,p}(\Omega)$.  This form, with an explicit but not optimal constant, is from \cite[section 7.8]{GilbargTrudinger2001}.

\begin{lemma} \label{lem:poincare}  If $\Omega\subset \RR^d$ is a bounded domain with volume $|\Omega|$, and if $1\le p<\infty$ then for all $u\in W_0^{1,p}(\Omega)$,
\begin{equation}
  \|u\|_{W^{1,p}(\Omega)}^p \le C(\Omega,p) \int_\Omega |\grad u|^p, \label{eq:poincare}
\end{equation}
where $C(\Omega,p)=1+(|\Omega|/\omega_d)^{p/d}$ and $\omega_d=(2 \pi^{d/2})/(d\,\Gamma(d/2))$ is the volume of the unit ball in $\RR^d$.
\end{lemma}

\section{Second-order Runge-Kutta time-discretization}  \label{app:rk2}  Section \ref{sec:strongform} describes the time semi-discretization of the continuum strong form \eqref{eq:massconserve}--\eqref{eq:constraint} using the $\theta$ method.  Such a one-stage method generates particular forms for the functions $\bQ_n(\bX,v,z)$ and $F_n(v,z)$ in equations \eqref{eq:semimassconserve}--\eqref{eq:semiconstraint}, and these functions then define weak formulation (VI) \eqref{eq:theVI}.  Here we illustrate how the corresponding functions $\bQ_n$ and $F_n$ can be generated for second-order Runge-Kutta (RK) schemes.

For the $m$-dimensional ODE system $\by' = \bg(t,\by)$ an $s$-stage RK scheme \cite{AscherPetzold1998} with time-step $h=\Delta t$ is given by constants $a_{ij},b_i,\tau_i$ and the equations
\begin{align}
  \by_{n,i} &= \by_{n-1} + h \sum_{j=1}^s a_{ij} \bg(t_{n-1} + \tau_j h, \by_{n,j}), \quad i=1,\dots,s \label{eq:RK2} \\
      \by_n &= \by_{n-1} + h \sum_{i=1}^s b_i \bg(t_{n-1} + \tau_i h, \by_{n,i}). \notag
\end{align}
\emph{Explicit} methods have $a_{ij}=0$ for $j\ge i$, i.e.~zeros on and above the diagonal in the Butcher tableau \cite{AscherPetzold1998}, while \emph{semi-implicit} methods have zeros above the diagonal.  Whereas general implicit RK schemes generate larger (nonlinear) systems, semi-implicit methods have the computational advantage that each stage generates an $m$-equation system.  Note that one must solve \eqref{eq:theVI} $s$ times to compute a time step using an $s$-stage explicit or semi-implicit RK scheme.

\emph{Diagonally-implicit} RK (DIRK) methods are semi-implicit methods for which the diagonal entries $a_{ii}$ are independent of $i$.  The accuracy of $s$-stage DIRK methods is limited to order $p=s+1$, and there exist strongly S-stable and stiffly-accurate \cite{AscherPetzold1998} DIRKs with order $p=s$ for $s=1,2,3$ \cite{Alexander1977}.  (``Strongly S-stable'' is also called ``stiff decay'' \cite{AscherPetzold1998}.)  The stability properties of these DIRK methods are helpful for mass conservation problems considered in the text, especially cases where $\bq$ has a leading-order diffusion term so that the $m$-dimensional method-of-lines ODE system is stiff.  In DIRK methods the linear system matrix can potentially be re-used at each stage.  (This matrix is $A = I - h a_{ii} J$ where the Jacobian $J$ is evaluated at the start of the time step, $J = \frac{\partial \bg}{\partial y}(t_{n-1},\by_{n-1})$.)

Now, as an illustration, we compute functions $\bQ_n$ and $F_n$ for two DIRK schemes.

\medskip
\renewcommand{\labelenumi}{\emph{(\alph{enumi})}}
\begin{enumerate}
\item The implicit midpoint rule is a $(s,p)=(2,2)$ A-stable DIRK scheme.  It uses a half backward Euler step followed by an explicit step:
\begin{align}
\tilde\by &= \by_{n-1} + \tfrac{1}{2} h \bg(t_{n-1}+\tfrac{1}{2}h,\tilde\by), \notag \\
\by_n &= \by_{n-1} + h \bg(t_{n-1}+\tfrac{1}{2}h,\tilde\by). \notag
\end{align}
Let $t_{n-1/2} = t_{n-1} + \tfrac{1}{2} \Delta t$.  Functions \eqref{eq:functionalforms} for the first stage are
  $$\tilde\bQ(\bX,v,x) = \tfrac{1}{2} \bq(\bX,v,x,t_{n-1/2}) \quad \text{and} \quad \tilde F(v,x) = \tfrac{1}{2} f(v,x,t_{n-1/2}).$$
Now let $\tilde u$ denote the weak solution to the first stage VI problem.  The functions for the explicit second stage are then $\bQ_n(\bX,v,x) = 0$ and
  $$\quad F_n(v,x) = f(\tilde u,x,t_{n-1/2}) - \Div \bq(\grad\tilde u,\tilde u,x,t_{n-1/2}).$$

\item The (unique) strongly S-stable $(s,p)=(2,2)$ scheme for which $0\le \tau_i\le 1$ \cite{AscherPetzold1998} has equations
\begin{align}
\tilde\by &= \by_{n-1} + \alpha h \bg(\tilde t,\tilde\by), \notag \\
\by_n &= \by_{n-1} + (1-\alpha) h \bg(\tilde t,\tilde\by) + \alpha h \bg(t_n,\by_n). \notag
\end{align}
where $\alpha = 1-\frac{\sqrt{2}}{2}$ and $\tilde t = t_{n-1} + \alpha h$.  Functions for the first stage are
  $$\tilde\bQ(\bX,v,x) = \alpha \bq(\bX,v,x,\tilde t) \quad \text{and} \quad \tilde F(v,x) = \alpha f(v,x,\tilde t).$$
If $\tilde u$ denotes the solution to the first stage VI then the functions for the second stage are $\bQ_n(\bX,v,x) = \alpha \bq(\bX,v,x,t_n)$ and
   $$F_n(v,x) = (1-\alpha) f(\tilde u,x,\tilde t) + \alpha f(v,x,t_n) - (1-\alpha) \Div \bq(\grad\tilde u,\tilde u,x,\tilde t).$$
\end{enumerate}

\end{document}